\documentclass[a4paper,11pt,twoside]{article}
\pdfoutput=1
\usepackage[T1]{fontenc}
\usepackage[english]{babel}
\usepackage{amsmath, amssymb, amsthm, amsfonts, mathtools, mathrsfs}	
\usepackage{enumerate}
\usepackage[shortlabels]{enumitem}
\usepackage{booktabs, caption, graphicx, subfig, wrapfig} 	
\usepackage{geometry}
\usepackage{algorithm}
\usepackage{algpseudocode} 
\usepackage{epstopdf} 
\usepackage{soul}

\usepackage{xcolor}

\usepackage{cite} 

\usepackage{fancyhdr} 

\usepackage{upgreek}
\usepackage{bbm}

\usepackage{microtype}

\newcommand{\R}{\mathbb{R}}
\newcommand{\N}{\mathbb{N}}

\numberwithin{equation}{section}

\theoremstyle{plain}
\newtheorem{theorem}{Theorem}[section] 
\newtheorem{proposition}[theorem]{Proposition} 
\newtheorem{lemma}[theorem]{Lemma}
\newtheorem{corollary}[theorem]{Corollary}

\theoremstyle{definition}

\newtheorem{remark}[theorem]{Remark}

\DeclarePairedDelimiter{\abs}{\lvert}{\rvert}
\DeclarePairedDelimiter{\norm}{\lVert}{\rVert}

\newcommand{\eps}{\varepsilon}
\renewcommand{\phi}{\varphi}
\renewcommand{\bar}{\overline}
\renewcommand{\vec}{\boldsymbol}

\def\genspazio #1#2#3#4#5{#1^{#2}(#5,#4;#3)}
\def\spazio #1#2#3{\genspazio {#1}{#2}{#3}T0}

\def\LT {\spazio L}
\def\HT {\spazio H}
\def\CT #1#2{C^{#1}([0,T];#2)}

\def\Lx #1{L^{#1}(\Omega)}
\def\Lt #1{L^{#1}(0,T)}
\def\Lqt #1{L^{#1}(Q_T)}
\def\Hx #1{H^{#1}(\Omega)}
\def\Hnx #1{H^{#1}_N(\Omega)}
\def\Wx #1{W^{#1}(\Omega)}

\def\Cqt #1{C^{#1}(\bar{Q_T})}

\def\Accorpa #1#2 #3 {\gdef #1{\eqref{#2}--\eqref{#3}}%
	\wlog{}\wlog{\string #1 -> #2 - #3}\wlog{}}%
	
\def\abramo #1{{\color{violet}#1}}

\def\ls{<}
\def\gs{>}
\def\mezzo {\frac{1}{2}}
\def\de {\mathrm{d}}
\def\D {\mathrm{D}}
\def\ddt {\frac{\de}{\de t}}

\def\n {\vec{n}}
\def\hh {\mathbbm{h}}
\def\phib {\bar{\phi}}
\def\mub {\bar{\mu}}
\def\sigmab {\bar{\sigma}}
\def\phiob {\bar{\phi_0}}
\def\sigmaob {\bar{\sigma_0}}

\def\Iad {\mathcal{I}_{\text{ad}}}

\def\Rcal {\mathcal{R}}

\def\phil {\phi_\lambda}
\def\psil {\psi_\lambda}
\def\phitil {\widetilde{\phi}}
\def\mutil {\widetilde{\mu}}
\def\sigmatil {\widetilde{\sigma}}

\usepackage{hyperref}
\usepackage{orcidlink}

\begin{document}
	
	\begin{center}
		
		{\LARGE \textbf{Carleman estimates for backward \\
		\vskip0.2cm
        Cahn--Hilliard-reaction-diffusion problems}}

		\vskip0.5cm
		
		{\large\textsc{Abramo Agosti$^1$} \orcidlink{0000-0001-5706-3772}} \\
		{\normalsize e-mail: \texttt{abramo.agosti@unipv.it}} \\
		\vskip0.35cm
		
		{\large\textsc{Elena Beretta$^2$} \orcidlink{0000-0002-4517-6945}} \\
		{\normalsize e-mail: \texttt{eb147@nyu.edu}} \\
		\vskip0.35cm
		
		{\large\textsc{Cecilia Cavaterra$^{3, 4}$} \orcidlink{0000-0002-2754-7714}} \\
		{\normalsize e-mail: \texttt{cecilia.cavaterra@unimi.it}} \\
		\vskip0.35cm
		
		{\large \textsc{Matteo Fornoni$^3$} \orcidlink{0000-0002-9787-047X}} \\
		{\normalsize e-mail: \texttt{matteo.fornoni@unimi.it}} \\
		\vskip0.35cm
		

		{\large\textsc{Masahiro Yamamoto$^{5}$}} \\
		{\normalsize e-mail: \texttt{myama@ms.u-tokyo.ac.jp}} \\
		\vskip0.35cm
		
		{\footnotesize $^1$Department of Mathematics ``F. Casorati'', University of Pavia, 27100 Pavia, Italy}
		\vskip0.1cm
		
		{\footnotesize $^2$Division of Science, New York University Abu Dhabi, Saadiyat Island, Abu Dhabi, United Arab Emirates}
		\vskip0.1cm
		
		{\footnotesize $^3$Department of Mathematics ``F. Enriques'', University of Milan, 20133 Milan, Italy}
		\vskip0.1cm

        {\footnotesize$^4$Istituto di Matematica Applicata e Tecnologie Informatiche ``Enrico Magenes'', CNR, 27100 Pavia, Italy}
        \vskip0.1cm

		{\footnotesize$^5$Department of Mathematical Sciences, The University of Tokyo, 153-8914 Tokyo, Japan}
        \vskip0.5cm
		
	\end{center}

    \begin{abstract}

    \noindent
    We study a backward inverse problem for a coupled Cahn--Hilliard-reaction-diffusion system. Our main result is a Carleman estimate for a fourth-order/second-order parabolic system with cross-diffusion terms, which allows us  to derive conditional stability estimates for the reconstruction of past states from a single final-time observation: H\"older stability at positive times and logarithmic stability for the initial datum.  We then apply the Carleman estimate to a phase-field tumour growth model coupling the tumour volume fraction with a nutrient concentration. In this setting, we obtain backward uniqueness and quantitative stability for the recovery of early tumour states, improving earlier results based on logarithmic convexity, which only yielded uniqueness under an additional smallness assumption on the chemotaxis coefficient. We also discuss how these results support Lipschitz stability on finite-dimensional admissible sets, which is relevant for ensuring convergence of iterative discretisation algorithms.

    \vskip3mm
		
    
    \vskip3mm
    

\end{abstract}


\pagestyle{fancy}
\fancyhf{}	
\fancyhead[EL]{\thepage}
\fancyhead[ER]{\textsc{Agosti -- Beretta -- Cavaterra -- Fornoni -- Yamamoto}} 
\fancyhead[OL]{\textsc{Carleman estimates for backward Cahn--Hilliard-reaction-diffusion}} 
\fancyhead[OR]{\thepage}

\renewcommand{\headrulewidth}{0pt}
\setlength{\headheight}{5mm}

\thispagestyle{empty} 

\section{Introduction}
\label{sec:intro}
	
Backward inverse problems for dissipative evolution equations are classical examples of severely ill-posed problems: even small perturbations in final observations may generate large errors when propagated backwards in time. Despite this intrinsic instability, such problems are of considerable interest in applications where only late-time measurements are available and the goal is to infer an earlier state of the system. This issue arises naturally in mathematical biology, and
particularly in tumour-growth modelling, where clinically accessible data are typically collected at diagnosis or during drug administration after a significant evolution period, while an important objective is to reconstruct a previous
stage of the disease. 
For instance, characterizing the initial tumour location, shape and extent can add important complementary diagnostic and prognostic indications with respect to standard informations based on biopsy and radiogenomics, being correlated to distinct tumour mutations and aggressiveness (see e.g. \cite{Roux,Li,SSMB2020}). This may lead to the attainment of improved prognostic indications for patient survival and to the design of more effective surgical intervention strategies. Evidences for the correlation between tumour location and shape with its grading patterns have been widely provided in literature, e.g. for gliomas \cite{Jungk,Chaichana} and prostate cancer \cite{Erbersdobler}. Also, for highly infiltrative tumours like gliomas, the identification of the tumour location at initiation or after surgery can help in  identifying possible infiltrative tumour regions which are typically not visible in clinical images, leading to the design of more effective radiation dose planning \cite{Rathore}. The reconstruction of previous tumour configurations, given an MRI scan, may be also crucial during therapy administration, leading to the possibility of identifying the tumour cells sensitivity to therapy and its resistant areas by reconstructing the tumour response backward in time. Determining the initial tumour configuration is also required for robustly identifying tumour growth parameters and calibrate the predictive tumour growth models in a patient-specific way, which may lead to accurate predictions of tumour evolution and recurrence (see e.g. \cite{Agosti1,Agosti2,Lorenzo2022_review}).

A powerful tool for analysing backward and inverse problems is provided by Carleman estimates, namely weighted integral inequalities that yield uniqueness, observability, and conditional stability for evolution equations. 
For parabolic problems, Carleman techniques have been extensively developed and successfully applied; we refer, for instance, to the foundational works of Yamamoto and Vessella \cite{yamamoto2009carleman, yamamoto, vessellabook, vessella2003}. However, in the specific context of backward Cahn--Hilliard-type equations, the available literature is much more limited. A closely related contribution, based on Carleman estimates, is the study by Shang and Li \cite{SL2021} on a backward fourth-order parabolic equation, establishing conditional H\"older stability for intermediate data  from final-time measurements. 

The present paper extends this line of research by considering a coupled Cahn--Hilliard and reaction-diffusion system with cross-diffusion terms. By exploiting the Carleman framework, we derive not only intermediate-time stability but also logarithmic estimates at the initial time. Furthermore, within the tumour-growth setting, we provide a refined Lipschitz stability result for finite-dimensional classes of admissible data, following an approach  similar to that in \cite{BCFLR2024}.

Motivated by these issues, we study the backward problem for the semilinear coupled system
\begin{align*}
\partial_t u + \Delta^2 u + a\Delta v &= F(x,t,u,v), \\
\partial_t v - \Delta v + b\Delta u &= G(x,t,u,v),
\end{align*}
subject to homogeneous Neumann boundary conditions. This model couples a fourth-order Cahn-Hilliard-type component with a second-order reaction-diffusion component and includes cross-diffusion effects through the coefficients $a,b>0$. Given the final measurement $(u(\cdot,T),v(\cdot,T))$, the inverse problem we address consists in determining earlier states $(u(\cdot,t_0),v(\cdot,t_0))$, with $0\leq t_0<T$. The abstract framework above is sufficiently general to encompass the phase-field tumour-growth model considered later in the paper.

Our first main result is a Carleman estimate for this backward coupled system. More precisely, we establish a weighted inequality for strong solutions which controls the time derivatives, the principal spatial terms, and the lower-order components of the solution in terms of the weighted source terms and explicit boundary contributions. A notable feature of this estimate is that no smallness restriction is imposed on the cross-diffusion coefficients $a$ and $b$. This aspect is essential for applications, since it shows that the method remains effective even in the presence of arbitrarily large cross-diffusion effects.
 The Carleman estimate is then applied to derive conditional stability results for the associated backward inverse problem. For every positive intermediate time $t_0\in(0,T)$, we prove a H\"older-type stability estimate recovering $(u(\cdot,t_0),v(\cdot,t_0))$ from the final data under a natural a priori bound on the initial state. Compared with the general strategy in \cite{yamamoto}, our argument yields a sharper form of the estimate by avoiding an additional logarithmic factor. Since the corresponding H\"older exponent degenerates as $t_0\to 0$, one cannot expect the same type of stability at the initial time; nevertheless, we show that the initial datum satisfies a logarithmic stability estimate. This transition from H\"older stability at positive times to a conditional logarithmic stability at time zero reflects the severe ill-posedness of backward parabolic problems.

A main motivation for this abstract analysis comes from phase-field models for tumour growth. In Section~4 we consider a Cahn--Hilliard and reaction-diffusion system of the form 
\begin{alignat*}{2}
    	& \partial_t \phi - \Delta \mu
    	= P(\phi) \left(\sigma + \chi (1-\phi) - \mu \right) - c(x,t) \hh (\phi)
    	\qquad && \hbox{in $Q_T$,} \\ 
    	& \mu 
    	= - \Delta \phi + F'(\phi) - \chi \sigma 
    	\qquad && \hbox{in $Q_T$,} \\ 
    	& \partial_t \sigma - \Delta \sigma + \chi \Delta \phi 
    	= - P(\phi) \left(\sigma + \chi (1-\phi) - \mu \right) + \kappa (1 - \sigma)
    	\qquad && \hbox{in $Q_T$,} \\ 
    	& \partial_{\n} \phi = \partial_{\n} \mu = \partial_{\n} \sigma = 0 
    	\qquad && \hbox{in $\Sigma_T$,} \\ 
    	& \phi(0) = \phi_0, \quad \sigma(0) = \sigma_0 
    	\qquad && \hbox{in $\Omega$,} 
    \end{alignat*}
in which the phase variable $\phi$ represents the tumour volume fraction, $\mu$ is the associated chemical potential, and $\sigma$ denotes the concentration of a nutrient driving tumour proliferation, while chemotaxis is encoded by a coefficient $\chi\ge 0$. This model was introduced in \cite{ABCFR2025} as a variant of the diffuse-interface multiphase mixture model originally developed in \cite{HPZO2013}. In the latter model the mass exchange terms between the phases were represented by chemical kinetic rates incorporating tumour cells proliferation and nutrient consumption. In the present model apoptosis and nutrient release by the healthy tissues are also incorporated in the source terms. Under standard assumptions ensuring strong well-posedness, the difference of two solutions can be rewritten in the same structural form as the abstract system above, with cross-diffusion coefficient equal to $\chi$. Therefore, the Carleman estimate developed in Sections~2 and~3 becomes directly applicable to the inverse problem of reconstructing earlier tumour states from a single final observation.
We stress that our general Carleman estimate also applies to other tumour growth models with similar structure, with possibly different reaction terms, as long as they admit a unique strong solution \cite{CSS2021, GL2017, FLOW2019, Fritz2023}.

This application extends the previous analysis in \cite{ABCFR2025}. In that work, backward uniqueness for the tumour growth model was obtained by means of a logarithmic convexity argument, but no quantitative stability estimate was derived and an additional smallness condition on the chemotactic coefficient $\chi$ was required. By contrast, the Carleman approach adopted here removes that extra restriction and yields explicit conditional H\"older stability for positive times, backward uniqueness, and logarithmic stability for the initial data. In this sense, Carleman estimates provide a broader and more informative framework than logarithmic convexity for the class of coupled fourth-/second-order systems considered in this paper.

Finally, we address the stability regime that is most relevant for reconstruction procedures, \cite{de2012local, BCFLR2025}. Indeed, when the admissible initial data are restricted to a finite-dimensional subspace, for instance a discrete space arising in numerical approximation, we combine the differentiability properties of the forward map with a quantitative injectivity estimate for the linearised problem and the general result from \cite{BV2006} to prove a Lipschitz stability estimate. This refined result is important from the numerical viewpoint, since it provides a rigorous foundation for reconstruction algorithms, even though the corresponding stability constants deteriorate with the dimension of the subspace and with the final time, consistently with the ill-posed character of the backward problem \cite{BCFLR2024, BCFLR2025}. The same strategy also suggests further developments for related phase-field tumour-growth models and for the analysis of reconstruction methods based on finite-dimensional approximations.

The paper is organised as follows. In Section~2 we prove the abstract Carleman estimate for backward Cahn--Hilliard-reaction-diffusion systems. Section~3 is devoted to the H\"older and logarithmic stability results for the associated backward inverse problem. In Section~4 we apply the abstract theory to the tumour growth model, obtain the corresponding uniqueness and stability results, and derive Lipschitz stability for initial states belonging to a finite dimensional subspace. Some final remarks and possible perspectives are collected in the concluding section.

\section{Carleman estimate}
\label{sec:carleman}

In this section, we consider a general semilinear fourth-order coupled parabolic system, and we derive a suitable Carleman estimate for the backward problem.
Let $\Omega \subset \R^d$, $d \in \N$, be an open bounded domain with smooth enough boundary $\partial \Omega$. 
Let $T > 0$ be a fixed final time, and set $Q_T := \Omega \times (0,T)$ and $\Sigma_T := \partial \Omega \times (0,T)$.
We consider the following semilinear parabolic system:
\begin{alignat}{2}
	& \partial_t u + \Delta^2 u + a \, \Delta v = F(x,t,u,v) \quad && \hbox{in $Q_T$,} \label{eq:u} \\
	& \partial_t v - \Delta v + b \, \Delta u = G(x,t,u,v) \quad && \hbox{in $Q_T$,} \label{eq:v}
\end{alignat}
endowed with the homogeneous Neumann boundary conditions
\begin{equation}
	\label{eq:bc}
	\partial_{\n} u = \partial_{\n} \Delta u = \partial_{\n} v = 0 \quad \hbox{on $\Sigma_T$,} 
\end{equation}
and with suitable initial conditions.
The coefficients $a,b \ge 0$ are given constants, promoting some cross-diffusion effects in the system, while $F,G: Q_T \times \Hx2 \times \Hx1 \to \R$ are two possibly non-linear Caratheodory functions, namely measurable in $(x,t) \in Q_T$ for every fixed $(u,v) \in \Hx2 \times \Hx1$ and continuous in $(u,v) \in \Hx2 \times \Hx1$ for almost every fixed $(x,t) \in Q_T$. Moreover, there exists a function $K \in \Lt \infty$, $K(t) \ge 0$ for a.e.~$t \in (0,T)$, such that, for almost every $t\in(0,T)$ and for any
$w_1\in H^2(\Omega)$, $w_2\in H^1(\Omega)$, the right-hand sides $F$ and $G$ satisfy
\begin{equation}
	\label{eq:hp_FG}
	\int_\Omega \left( \abs{F(x,t,w_1,w_2)}^2 + \abs{G(x,t,w_1,w_2)}^2 \right) \, \de x \le K(t) \left( \norm{w_1}^2_{\Hx2} + \norm{w_2}^2_{\Hx1} \right).
\end{equation}
We further assume that $(u,v)$ is a strong solution to the system \eqref{eq:u}--\eqref{eq:bc}, satisfying the equations almost everywhere in $Q_T$ and enjoying the regularity
\begin{equation}
	\label{eq:hp_reg}
	\begin{split}
		& u \in \HT 1 {\Lx2} \cap \CT 0 {\Hx2} \cap \LT 2 {\Hx4}, \\
		& v \in \HT 1 {\Lx2} \cap \CT 0 {\Hx1} \cap \LT 2 {\Hx2}. 
	\end{split}
\end{equation}
We are interested in a Carleman estimate to address the following backward-in-time inverse problem: given a spatial measurement of the solution $(u,v)$ at the final time $t = T$, i.e., $(u(\cdot,T), v(\cdot,T)) \in \Hx2 \times \Hx1$, we ask whether we can uniquely determine the earlier states $(u(\cdot,t_0), v(\cdot,t_0)) \in \Hx2 \times \Hx1$ for some $t_0 \in [0,T)$.
To this end, given a parameter $\lambda > 0$, we introduce the weight functions 
\begin{equation}
	\label{eq:hp_weight1}
	\phil \in C^2([0,T]), \quad \psil \in C^0([0,T]), \quad \psil(t) > 0 \quad \hbox{for any $t \in [0,T]$},
\end{equation}
such that there exist constants $L_1, L_2 > 0$, independent of $\lambda$, but possibly depending on $T$, for which
\begin{equation}
	\label{eq:hp_weight2}
	\abs{\phil'(t)} \le L_1 \lambda \psil(t), \quad 
	\phil''(t) \ge L_2\lambda^2 \psil(t), \quad 
	\hbox{for any $t \in [0,T]$.}
\end{equation}
Some possible choices for the weight functions $\phil$ and $\psil$ satisfying \eqref{eq:hp_weight1}--\eqref{eq:hp_weight2} are given by
\begin{itemize}
    \item \emph{Exponential weights:} $\phil(t) = \psil(t) = e^{\lambda t}$ for $\lambda > 0$, with $L_1 = L_2 =1$.
    \item \emph{Polynomial weights:} $\phil(t) = \psil(t) = (t+1)^\lambda$ for $\lambda \geq 2$. Indeed, in this case, for any $t \in [0,T]$, we have that 
    \[
        \abs{\phil'(t)} = \abs{\lambda (t+1)^{\lambda - 1}} = \abs*{\frac{\lambda}{t+1}} \phil(t) \le \lambda \phil(t),
    \]
    and that
    \[
        \phil''(t) = \lambda (\lambda - 1) (t+1)^{\lambda - 2} = \lambda (\lambda - 1) \frac{1}{(t+1)^2} \phil(t) \ge \frac{\lambda^2}{2} \frac{1}{(T+1)^2} \phil(t),
    \]
    where we used that $\lambda - 1 \ge \frac{\lambda}{2}$ for any $\lambda \geq 2$.
    Thus, \eqref{eq:hp_weight2} holds with $L_1 = 1$ and $L_2 = \frac{1}{2(T+1)^2}$.
\end{itemize}
Then, we have the following result.

\begin{theorem}
	\label{thm:carleman}
	Let $(u,v)$ be a strong solution to \eqref{eq:u}--\eqref{eq:bc}, and assume \eqref{eq:hp_FG}--\eqref{eq:hp_reg}. 
	For any $\lambda > 0$, let $\phil$ and $\psil$ be two weight functions satisfying \eqref{eq:hp_weight1}--\eqref{eq:hp_weight2}.

	Then, there exist constants $C > 0$, depending only on the parameters of the system and on $T$, and $\lambda_0, s_0 > 0$, such that for any $\lambda > \lambda_0$ and $s > s_0$, the following estimate holds:
    \begin{equation}
		\label{eq:carleman}
		\begin{split}
			& \int_{Q_T} \left\{ \frac{1}{s \psil(t)} \left( \abs{\partial_t u}^2 + \abs{\partial_t v}^2 \right) + \lambda \left( \abs{\Delta u}^2 + \abs{\nabla v}^2 \right) + s \lambda^2 \psil(t) \left( \abs{u}^2 + \abs{v}^2 \right) \right\} e^{2s\phil(t)} \, \de x \, \de t \\
			& \quad \le C \left( \int_{Q_T} \left( \abs{F(x,t,u,v)}^2 + \abs{G(x,t,u,v)}^2 \right) e^{2s\phil(t)} \, \de x \, \de t 
			+ \sum_{k=1}^7 \abs{B_k} \right).
		\end{split}
	\end{equation}
	where the boundary terms $B_k$, $k=1,\dots,7$, are given by
	\begin{align*}
		& B_1 := s \int_\Omega \left[ \phil'(t) \abs{u}^2 e^{2s\phil(t)} \right]^0_T \, \de x, 
		&& B_2 := \int_\Omega \left[ \abs{\Delta u}^2 e^{2s \phil(t)} \right]^T_0 \, \de x, \\
		& B_3 := s \int_\Omega \left[ \phil'(t) \abs{v}^2 e^{2s\phil(t)} \right]^0_T \, \de x, 
		&& B_4 := \int_\Omega \left[ \abs{\nabla v}^2 e^{2s \phil(t)} \right]^T_0 \, \de x, \\
		& B_5 := 2 \int_\Omega \left[ \abs{\nabla u \cdot \nabla v} e^{2s \phil(t)} \right]^T_0 \, \de x, 
		&& B_6 := \frac{\lambda}{2} \int_\Omega \left[ \abs{u}^2 e^{2s\phil(t)} \right]^0_T \, \de x, \\
		& B_7 := \frac{\lambda}{2} \int_\Omega \left[ \abs{v}^2 e^{2s\phil(t)} \right]^0_T \, \de x.
	\end{align*}
\end{theorem}

\begin{remark}
    Estimate \eqref{eq:carleman} is an inequality of Carleman type, which estimates $u$, $v$ and the derivatives uniformly in large parameters $s>0$, $\lambda>0$. Thanks to such large parameters as factors on the left-hand side, the estimate \eqref{eq:carleman} holds also for a system of semilinear parabolic inequalities
    \begin{alignat*}{2}
	   & \abs{\partial_t u + \Delta^2 u + a \, \Delta v} \le R_1(x,t,u,v) \quad && \hbox{in $Q_T$,} \\
	   & \abs{\partial_t v - \Delta v + b \, \Delta u} \le R_2(x,t,u,v) \quad && \hbox{in $Q_T$,}
    \end{alignat*}
    satisfying the boundary conditions \eqref{eq:bc}, with $R_1$ and $R_2$ non-negative and satisfying the bounds \eqref{eq:hp_FG}. 
    Then, all the results below can be proved also in this more general setting. 
    For more details on general Carleman estimates we refer the reader to \cite{isakov2006inverse, yamamoto, vessellabook}.

    Moreover, we observe that the decomposition $b\Delta u$ in \eqref{eq:v} is not unique because $G$ can contain terms of the form $\Delta u$, due to hypothesis \eqref{eq:hp_FG}. However, estimate \eqref{eq:carleman} holds independently of the decomposition, up to changing the constant $C$.
\end{remark}

\begin{proof}
	We start with a change of variables, by setting
	\begin{equation}
		\label{eq:changevar}
		y(x,t) := u(x,t) e^{s \phil(t)}, \quad z(x,t) := v(x,t) e^{s \phil(t)}, \quad \hbox{for any $(x,t) \in Q_T$,}
	\end{equation}
	for a given parameter $s > 0$, to be chosen large enough later on.
	Then, since $\phil$ depends only on time, we can rewrite the system \eqref{eq:u}--\eqref{eq:bc} in terms of the new variables $(y,z)$ as
	\begin{alignat}{2}
	& \partial_t y - s (\phil'(t)) \, y + \Delta^2 y + a \, \Delta z = F\!\left(x,t,y e^{- s \phil(t)},z e^{- s \phil(t)}\right) \, e^{s \phil(t)} \quad && \hbox{a.e.~in $Q_T$,} \label{eq:y} \\
	& \partial_t z  - s (\phil'(t)) \, z - \Delta z + b \, \Delta y = G\!\left(x,t,y e^{- s \phil(t)},z e^{- s \phil(t)}\right) \, e^{s \phil(t)} \quad && \hbox{a.e.~in $Q_T$,} \label{eq:z}
	\end{alignat}
	endowed with the homogeneous Neumann boundary conditions
	\begin{equation}
		\label{eq:bcyz}
		\partial_{\n} y = \partial_{\n} \Delta y = \partial_{\n} z = 0 \quad \hbox{a.e.~on $\Sigma_T$.} 
	\end{equation}
	We follow the approach used in \cite[Section 9]{yamamoto2009carleman} and \cite[Chapter 7]{yamamoto} for a similar Carleman estimate on a single second-order parabolic equation. 
	However, in our case, we also have to treat the cross-diffusion terms, which require some additional care. Indeed, we start by analysing the more general case in which both cross-diffusion coefficients $a$ and $b$ are strictly positive.

	\textsc{First Estimate.} 
	As \eqref{eq:y} holds in the strong sense, we start by taking the $L^2(Q_T)$-norm on both sides of the equation, obtaining 
	\[
		\norm{F e^{s \phil}}^2_{\Lqt2} = \norm{\partial_t y + (- s \phil' y + \Delta^2 y + a \, \Delta z)}^2_{\Lqt2}.
	\]
	Then, by expanding the squared norm on the right-hand side and by neglecting the last non-negative term, we infer that
	\begin{align*}
		\norm{F e^{s \phil}}^2_{\Lqt2} & \ge \norm{\partial_t y}^2_{\Lqt2} + 2(\partial_t y,(- s \phil' y + \Delta^2 y + a \, \Delta z))_{\Lqt2} \\
		& \ge \norm{\partial_t y}^2_{\Lqt2} 
		- 2 \int_0^T \int_\Omega s \phil'(t) y \, \partial_t y \, \de x \, \de t \\
		& \quad + 2 \int_0^T \int_\Omega \partial_t y \, \Delta^2 y \, \de x \, \de t 
		+ 2 a \int_0^T \int_\Omega \partial_t y \, \Delta z \, \de x \, \de t \\
		& \ge \norm{\partial_t y}^2_{\Lqt2} + I_1 + I_2 + 2 a \int_0^T \int_\Omega \partial_t y \, \Delta z \, \de x \, \de t.
	\end{align*}
	Recalling that $y \in \HT 1 {\Lx2} \cap C^0([0,T]; \Hx2) \cap \LT 2 {\Hx4}$, we can use standard identities and integration by parts to deduce that 
	\begin{align*}
		I_1 & = - 2s \int_0^T \int_\Omega \phil'(t) y \, \partial_t y \, \de x \, \de t = - s \int_0^T \phil'(t) \ddt \left( \int_\Omega \abs{y}^2 \, \de x \right) \, \de t \\
		& = s \int_0^T \int_\Omega \phil''(t) \abs{y}^2 \, \de x \, \de t + s \int_\Omega \left[ \phil'(t) \abs{y}^2 \right]^0_T \, \de x \\
		& = s \int_0^T \int_\Omega \phil''(t) \abs{y}^2 \, \de x \, \de t + B_1, \\
		I_2 & = 2 \int_0^T \int_\Omega \partial_t y \, \Delta^2 y \, \de x \, \de t 
		= \int_0^T \ddt \left( \int_\Omega \abs{\Delta y}^2 \, \de x \right) \, \de t \\
		& = \int_\Omega \left[ \abs{\Delta y}^2 \right]^T_0 \, \de x = B_2, 
	\end{align*} 
	where $B_1$ and $B_2$ are the boundary terms defined in Theorem \ref{thm:carleman}, up to the change of variables \eqref{eq:changevar}.
	Thus, we have shown that 
	\begin{equation}
		\label{eq:carleman:est1}
		\norm{F e^{s \phil}}^2_{\Lqt2} \ge \norm{\partial_t y}^2_{\Lqt2} + s \int_{Q_T} \phil''(t) \abs{y}^2 \, \de x \, \de t + 2a \int_{Q_T} \partial_t y \, \Delta z \, \de x \, \de t + B_1 + B_2.
	\end{equation}
	Next, we proceed similarly with \eqref{eq:z}. 
	Indeed, by taking the $L^2(Q_T)$-norm on both sides and by arguing as above, we obtain that 
	\begin{align*}
		\norm{G e^{s \phil}}^2_{\Lqt2} & \ge \norm{\partial_t z}^2_{\Lqt2} + 2(\partial_t z,(- s \phil' z - \Delta z + b \, \Delta y))_{\Lqt2} \\
		& \ge \norm{\partial_t z}^2_{\Lqt2} 
		- 2 \int_0^T \int_\Omega s \phil'(t) z \, \partial_t z \, \de x \, \de t \\
		& \quad - 2 \int_0^T \int_\Omega \partial_t z \, \Delta z \, \de x \, \de t 
		+ 2 b \int_0^T \int_\Omega \partial_t z \, \Delta y \, \de x \, \de t \\
		& \ge \norm{\partial_t z}^2_{\Lqt2} + I_3 + I_4 + 2 b \int_0^T \int_\Omega \partial_t z \, \Delta y \, \de x \, \de t.
	\end{align*} 
	Then, by recalling that $z \in \HT 1 {\Lx2} \cap C^0([0,T]; \Hx1) \cap \LT 2 {\Hx2}$, we can use again standard identities and integration by parts to deduce that 
	\begin{align*}
		I_3 & = - 2s \int_0^T \int_\Omega \phil'(t) z \, \partial_t z \, \de x \, \de t = - s \int_0^T \phil'(t) \ddt \left( \int_\Omega \abs{z}^2 \, \de x \right) \, \de t \\
		& = s \int_0^T \int_\Omega \phil''(t) \abs{z}^2 \, \de x \, \de t + s \int_\Omega \left[ \phil'(t) \abs{z}^2 \right]^0_T \, \de x \\
		& = s \int_0^T \int_\Omega \phil''(t) \abs{z}^2 \, \de x \, \de t + B_3, \\
		I_4 & = - 2 \int_0^T \int_\Omega \partial_t z \, \Delta z \, \de x \, \de t 
		= \int_0^T \ddt \left( \int_\Omega \abs{\nabla z}^2 \, \de x \right) \, \de t \\
		& = \int_\Omega \left[ \abs{\nabla z}^2 \right]^T_0 \, \de x = B_4, 
	\end{align*} 
	where $B_3$ and $B_4$ are the boundary terms defined in Theorem \ref{thm:carleman}, up to the change of variables \eqref{eq:changevar}. 
	Thus, we deduce that 
	\begin{equation}
		\label{eq:carleman:est2}
		\norm{G e^{s \phil}}^2_{\Lqt2} \ge \norm{\partial_t z}^2_{\Lqt2} + s \int_{Q_T} \phil''(t) \abs{z}^2 \, \de x \, \de t + 2b \int_{Q_T} \partial_t z \, \Delta y \, \de x \, \de t + B_3 + B_4.
	\end{equation}
	Next, since we are initially assuming that $a,b > 0$, we multiply \eqref{eq:carleman:est1} by $1/a$ and \eqref{eq:carleman:est2} by $1/b$, and we sum the resulting inequalities to obtain
	\begin{align*}
		& \frac{1}{a} \norm{\partial_t y}^2_{\Lqt2} +  \frac{1}{b} \norm{\partial_t z}^2_{\Lqt2} 
		+ \int_{Q_T} s \phil''(t) \left( \frac{1}{a} \abs{y}^2 + \frac{1}{b} \abs{z}^2 \right) \, \de x \, \de t \\
		& \qquad + 2 \int_{Q_T} \left( \partial_t y \Delta z + \partial_t z \Delta y \right) \, \de x \, \de t 
		+ \frac{1}{a} (B_1 + B_2) + \frac{1}{b} (B_3 + B_4) \\
		& \quad \le \frac{1}{a} \norm{F e^{s \phil}}^2_{\Lqt2} + \frac{1}{b} \norm{G e^{s \phil}}^2_{\Lqt2}.
	\end{align*}
	At this point, we observe that, since $y,z \in \HT 1 {\Lx2} \cap C^0([0,T]; \Hx1) \cap \LT 2 {\Hx2}$ and satisfy \eqref{eq:bcyz}, the following identity holds by integration by parts:
	\[
		2 \int_0^T \int_\Omega \left( \partial_t y \Delta z + \partial_t z \Delta y \right) \, \de x \, \de t 
		= 2 \int_0^T \ddt \left( \int_\Omega \nabla y \cdot \nabla z \, \de x \right) \, \de t 
		= 2 \int_\Omega \left[ \nabla y \cdot \nabla z \right]^T_0 \, \de x = B_5,
	\]
	where $B_5$ is the boundary term defined in Theorem \ref{thm:carleman}, up to the change of variables \eqref{eq:changevar}.
	Hence, by factoring $\min\{a^{-1},b^{-1}\}$ on the left-hand side and multiplying the above inequality by it, we infer that 
	\begin{equation*}
		\begin{split}
			& \norm{\partial_t y}^2_{\Lqt2} +  \norm{\partial_t z}^2_{\Lqt2} 
			+ \int_{Q_T} s \phil''(t) \left( \abs{y}^2 + \abs{z}^2 \right) \, \de x \, \de t \\
			& \quad \le C \left( \norm{F e^{s \phil}}^2_{\Lqt2} + \norm{G e^{s \phil}}^2_{\Lqt2} \right)
			+ C \sum_{k=1}^5 \abs{B_k},
		\end{split}
	\end{equation*}
	for some constant $C > 0$ depending only on $a$ and $b$.
	Moreover, by using the second assumption in \eqref{eq:hp_weight2} and by factoring $\min\{1,L_2\}$, we finally conclude that
	\begin{equation}
		\label{eq:carleman:first}
		\begin{split}
			& \norm{\partial_t y}^2_{\Lqt2} +  \norm{\partial_t z}^2_{\Lqt2} 
			+ \int_{Q_T} s \lambda^2 \psil(t) \left( \abs{y}^2 + \abs{z}^2 \right) \, \de x \, \de t \\
			& \quad \le C \left( \norm{F e^{s \phil}}^2_{\Lqt2} + \norm{G e^{s \phil}}^2_{\Lqt2} \right)
			+ C \sum_{k=1}^5 \abs{B_k},
		\end{split}
	\end{equation}
	for some constant $C>0$ now depending only on $a$, $b$, and possibly on $T$.

	\textsc{Second Estimate.}
	We now test \eqref{eq:y} by $y$ and integrate over $Q_T$ to obtain that 
	\begin{align*}
		& \mezzo \int_0^T \ddt \left( \int_\Omega \abs{y}^2 \, \de x \right) \, \de t 
		- \int_{Q_T} s \phil'(t) \abs{y}^2 \, \de x \, \de t \\
		& \quad + \int_{Q_T} \abs{\Delta y}^2 \, \de x \, \de t 
		+ a \int_{Q_T} z \Delta y \, \de x \, \de t 
		= \int_{Q_T} F e^{s \phil} y \, \de x \, \de t, 
	\end{align*}
	where we integrated by parts two times in the third and fourth term. 
	Then, by multiplying everything by $\lambda > 0$, rearranging the terms and applying Young's inequality, we infer that 
	\begin{align*}
		& \lambda \int_{Q_T} \abs{\Delta y}^2 \, \de x \, \de t \\
		& \quad = - \frac{\lambda}{2} \int_\Omega \left[ \abs{y}^2 \right]^T_0 \, \de x 
		+ s \int_{Q_T} \lambda \phil'(t) \abs{y}^2 \, \de x \, \de t
		- a \int_{Q_T} \lambda z \Delta y \, \de x \, \de t 
		+ \int_{Q_T} \lambda F e^{s \phil} y \, \de x \, \de t \\
		& \quad \le - \frac{\lambda}{2} \int_\Omega \left[ \abs{y}^2 \right]^T_0 \, \de x 
		+ s \int_{Q_T} \lambda \phil'(t) \abs{y}^2 \, \de x \, \de t
		+ \frac{\lambda}{2} \int_{Q_T} \abs{\Delta y}^2 \, \de x \, \de t \\
		& \qquad + \frac{a^2}{2} \int_{Q_T} \lambda \abs{z}^2 \, \de x \, \de t
		+ \mezzo \int_{Q_T} F^2 e^{2s \phil} \, \de x \, \de t + \mezzo \int_{Q_T} \lambda^2 \abs{y}^2 \, \de x \, \de t. 
	\end{align*}
	By recalling the definition of the boundary term $B_6$, up to the change of variables \eqref{eq:changevar}, and by employing the first hypothesis in \eqref{eq:hp_weight2}, the inequality above yields 
	\begin{align*}
		& \frac{\lambda}{2} \int_{Q_T} \abs{\Delta y}^2 \, \de x \, \de t \\
		& \quad \le \abs{B_6} + \int_{Q_T} \left(s L_1 \lambda^2 \psil(t) + \frac{\lambda^2}{2} \right) \abs{y}^2 \, \de x \, \de t 
		+ \frac{a^2}{2} \int_{Q_T} \lambda \abs{z}^2 \, \de x \, \de t
		+ \mezzo \norm{F e^{s \phil}}^2_{\Lqt2}.
	\end{align*}
	Finally, by choosing $\lambda > 0$ and $s > 0$ large enough such that 
	\begin{equation}
		\label{eq:sl_large}
		s \lambda \psil(t) \ge 1 \quad \hbox{for any $t \in [0,T]$},
	\end{equation}
	we can group the second and third term on the right-hand side to deduce that 
	\begin{equation}
		\label{eq:carleman:est3}
		\lambda \int_{Q_T} \abs{\Delta y}^2 \, \de x \, \de t \le C \int_{Q_T} s \lambda^2 \psil(t) \left( \abs{y}^2 + \abs{z}^2 \right) \, \de x \, \de t 
		+ \norm{F e^{s \phil}}^2_{\Lqt2}
		+ 2 \abs{B_6},
	\end{equation}
	for some constant $C > 0$, depending only on $a$ and possibly on $T$.
	Next, we proceed similarly on equation \eqref{eq:z} by testing it by $z$ and integrating over $Q_T$. 
	Indeed, by also performing some integrations by parts, we get that
	\begin{align*}
		& \mezzo \int_0^T \ddt \left( \int_\Omega \abs{z}^2 \, \de x \right) \, \de t 
		- \int_{Q_T} s \phil'(t) \abs{z}^2 \, \de x \, \de t \\
		& \quad + \int_{Q_T} \abs{\nabla z}^2 \, \de x \, \de t 
		+ b \int_{Q_T} z \Delta y \, \de x \, \de t 
		= \int_{Q_T} G e^{s \phil} z \, \de x \, \de t, 
	\end{align*}
	Then, by multiplying everything by $\lambda > 0$, rearranging the terms and applying Young's inequality, we deduce that
	\begin{align*}
		& \lambda \int_{Q_T} \abs{\nabla z}^2 \, \de x \, \de t \\
		& \quad = - \frac{\lambda}{2} \int_\Omega \left[ \abs{z}^2 \right]^T_0 \, \de x 
		+ s \int_{Q_T} \lambda \phil'(t) \abs{z}^2 \, \de x \, \de t
		- b \int_{Q_T} \lambda z \Delta y \, \de x \, \de t 
		+ \int_{Q_T} \lambda G e^{s \phil} z \, \de x \, \de t \\
		& \quad \le - \frac{\lambda}{2} \int_\Omega \left[ \abs{z}^2 \right]^T_0 \, \de x 
		+ s \int_{Q_T} \lambda \phil'(t) \abs{z}^2 \, \de x \, \de t
		+ \lambda \int_{Q_T} \abs{\Delta y}^2 \, \de x \, \de t \\
		& \qquad + \frac{b^2}{4} \int_{Q_T} \lambda \abs{z}^2 \, \de x \, \de t
		+ \mezzo \int_{Q_T} G^2 e^{2s \phil} \, \de x \, \de t + \mezzo \int_{Q_T} \lambda^2 \abs{z}^2 \, \de x \, \de t. 
	\end{align*}
	As we did above, we now incorporate the definition of the boundary term $B_7$, up to the change of variables \eqref{eq:changevar}, and we employ hypothesis \eqref{eq:hp_weight2} to obtain 
	\begin{align*}
		& \lambda \int_{Q_T} \abs{\nabla z}^2 \, \de x \, \de t \\
		& \quad \le \abs{B_7} 
		+ \int_{Q_T} \left( sL_1 \lambda^2 \psil(t) + \frac{b^2}{4} \lambda + \frac{\lambda^2}{2} \right) \abs{z}^2 \, \de x \, \de t 
		+ \lambda \int_{Q_T} \abs{\Delta y}^2 \, \de x \, \de t 
		+ \mezzo \norm{G e^{s \phil}}^2_{\Lqt2}.
	\end{align*}
	Next, by choosing again $s > 0$ and $\lambda > 0$ large enough so that \eqref{eq:sl_large} holds, we infer that
	\begin{equation}
		\label{eq:carleman:est4}
		\begin{split}
		& \lambda \int_{Q_T} \abs{\nabla z}^2 \, \de x \, \de t 
		- \lambda \int_{Q_T} \abs{\Delta y}^2 \, \de x \, \de t \\
		& \quad \le C \int_{Q_T} s \lambda^2 \psil(t) \abs{z}^2 \, \de x \, \de t 
		+ \mezzo \norm{G e^{s \phil}}^2_{\Lqt2} 
		+ \abs{B_7}, 
		\end{split}
	\end{equation}
	for some constant $C > 0$, depending only on $b$ and possibly on $T$.
	Therefore, by summing up \eqref{eq:carleman:est3} and \eqref{eq:carleman:est4} multiplied by $1/2$, we get that
	\begin{align*}
		& \frac{\lambda}{2} \int_{Q_T} \left( \abs{\Delta y}^2 + \abs{\nabla z}^2 \right) \, \de x \, \de t \le C \int_{Q_T} s \lambda^2 \psil(t) \left( \abs{y}^2 + \abs{z}^2 \right) \, \de x \, \de t \\
		& \qquad + C \left( \norm{F e^{s \phil}}^2_{\Lqt2} + \norm{G e^{s \phil}}^2_{\Lqt2} \right) 
		+ C \left( \abs{B_6} + \abs{B_7} \right),
	\end{align*}
	for some constant $C > 0$, depending only on $a$, $b$ and possibly on $T$.
	Finally, by using \eqref{eq:carleman:first} to estimate the first term on the right-hand side, we conclude that 
	\begin{equation}
		\label{eq:carleman:second}
		\begin{split}
			& \int_{Q_T} \lambda \left( \abs{\Delta y}^2 + \abs{\nabla z}^2 \right) \, \de x \, \de t \\
			& \quad \le C \left( \norm{F e^{s \phil}}^2_{\Lqt2} + \norm{G e^{s \phil}}^2_{\Lqt2} \right)
			+ C \sum_{k=1}^7 \abs{B_k}.
		\end{split}
	\end{equation}

	\textsc{Third Estimate.} 
	To deduce \eqref{eq:carleman}, we are just left to estimate the two terms containing the time-derivatives.
	Indeed, since $y = u e^{s \phil}$ by \eqref{eq:changevar}, we have that 
	\[
		\partial_t y = \partial_t u \, e^{s \phil} + s \phil' u e^{s \phil} \quad \hbox{a.e. in $Q_T$.}
	\]
	Then, by taking the squared absolute value on both sides, we get
	\begin{align*}
		\abs{\partial_t y}^2 & = e^{2s \phil} \abs{\partial_t u + s \phil' u}^2 = e^{2s \phil} \left( \abs{\partial_t u}^2 + 2s \phil' \partial_t u \, u + (s \phil')^2 \abs{u}^2 \right) \\
		& \ge e^{2s \phil} \left( \abs{\partial_t u}^2 - 2s \phil' \abs{\partial_t u} \, \abs{u} + (s \phil')^2 \abs{u}^2 \right)
	\end{align*}
	To treat the middle term, we employ Young's inequality 
	$2ab \le \eps a^2 + (1/\eps) b^2$ with $a=\abs{\partial_t u}$, $b = s \phil' \abs{u}$ and $\eps = 1/2$, to obtain
	\begin{align*}
		\abs{\partial_t y}^2 & \ge e^{2s \phil} \left( \abs{\partial_t u}^2 - \frac12 \abs{\partial_t u}^2 - 2 (s \phil')^2 \abs{u}^2 + (s \phil')^2 \abs{u}^2 \right) \\
		& = e^{2s \phil} \left( \frac12 \abs{\partial_t u}^2 - (s \phil')^2 \abs{u}^2 \right).
	\end{align*} 
	Next, by reverting the inequality, multiplying both sides by $1/(s \psil(t)) > 0$, integrating over $Q_T$ and using hypothesis \eqref{eq:hp_weight2}, we infer that 
	\begin{align*}
		\int_{Q_T} \frac{1}{s \psil(t)} \abs{\partial_t u}^2 e^{2s \phil} \, \de x \, \de t 
		& \le 2 \int_{Q_T} \frac{s^2 \phil'(t)^2}{s \psil(t)} \abs{u}^2 e^{2s \phil} \, \de x \, \de t 
		+ 2 \int_{Q_T} \frac{1}{s \psil(t)} \abs{\partial_t y}^2 \, \de x \, \de t \\
		& \le 2L_1^2 \int_{Q_T} s \lambda^2 \psil(t) \abs{u}^2 e^{2s \phil} \, \de x \, \de t 
		+ 2 \int_{Q_T} \frac{1}{s \psil(t)} \abs{\partial_t y}^2 \, \de x \, \de t \\
		& \le C \int_{Q_T} s \lambda^2 \psil(t) \abs{y}^2 \, \de x \, \de t + C \int_{Q_T} \abs{\partial_t y}^2 \, \de x \, \de t \\
		& \le C \left( \norm{F e^{s \phil}}^2_{\Lqt2} + \norm{G e^{s \phil}}^2_{\Lqt2} \right)
		+ C \sum_{k=1}^5 \abs{B_k},
	\end{align*}
	where in the second to last line we used \eqref{eq:changevar} and the fact that $s \psil(t) > 1$ for $s$ large enough, whereas in the last line we used \eqref{eq:carleman:first}. 
	By arguing in a completely similar fashion also for $z = v e^{s \phil}$, we finally conclude that 
	\begin{equation}
		\label{eq:carleman:third}
		\begin{split}
			& \int_{Q_T} \frac{1}{s \psil(t)} \left( \abs{\partial_t u}^2 + \abs{\partial_t v}^2 \right) e^{2s \phil} \, \de x \, \de t \\
			& \quad \le C \left( \norm{F e^{s \phil}}^2_{\Lqt2} + \norm{G e^{s \phil}}^2_{\Lqt2} \right)
			+ C \sum_{k=1}^5 \abs{B_k}.
		\end{split}
	\end{equation}
	Hence, by putting together \eqref{eq:carleman:first}, \eqref{eq:carleman:second} and \eqref{eq:carleman:third}, up to the change of variables \eqref{eq:changevar}, we deduce the Carleman estimate \eqref{eq:carleman}.

    We now briefly discuss how to extend the result to the additional cases in which one or both cross-diffusion terms degenerate.

    \textbf{Case $a,b = 0$.} The proof can be easily extended to this case, due to the absence of the cross-diffusion terms. Indeed, by repeating the same process, one can prove \eqref{eq:carleman:first}, \eqref{eq:carleman:second} and \eqref{eq:carleman:third}, simply without the appearance of the boundary term $B_5$. 

    \textbf{Case $a = 0$, $b > 0$.} In this case, we observe that hypothesis \eqref{eq:hp_FG} allows us to include the cross-diffusion term $b \Delta u$ in \eqref{eq:v} on the right-hand side $G$. 
    Thus, we consider the system:
    \begin{alignat*}{2}
	   & \partial_t u + \Delta^2 u = F(x,t,u,v) \quad && \hbox{in $Q_T$,} \\
	   & \partial_t v - \Delta v = \widetilde{G}(x,t,u,v) \quad && \hbox{in $Q_T$,}
    \end{alignat*}
    where 
    \[
        \widetilde{G}(x,t,u,v) = G(x,t,u,v) + b \Delta u.
    \]
    In this way, the new source term $\widetilde{G}$ still satisfies the main hypothesis \eqref{eq:hp_FG}, up to possibly changing $K(t) \in \Lt\infty$, now depending also on $b$. 
    Then, the Carleman estimate \eqref{eq:carleman} follows by the general argument above with $a = 0$ and $b = 0$, without the boundary term $B_5$.

    \textbf{Case $a > 0$, $b = 0$.} In this situation, we cannot proceed as in the previous case, since we only have estimates in $\Hx1$ for $v$. Thus, we modify \eqref{eq:u}--\eqref{eq:v} by adding a term $+ \Delta u$ on both sides of \eqref{eq:v} instead, namely
    \begin{alignat*}{2}
	   & \partial_t u + \Delta^2 u + a \, \Delta v = F(x,t,u,v) \quad && \hbox{in $Q_T$,} \\
	   & \partial_t v - \Delta v + \Delta u = \widetilde{G}(x,t,u,v) \quad && \hbox{in $Q_T$,}
    \end{alignat*}
    where 
    \[
        \widetilde{G}(x,t,u,v) = G(x,t,u,v) + \Delta u.
    \]
    In this way, the new source term $\widetilde{G}$ similarly satisfies the main hypothesis \eqref{eq:hp_FG}, up to possibly changing $K(t) \in \Lt\infty$. 
    Thus, the Carleman estimate \eqref{eq:carleman} follows by the general argument above with $a > 0$ and $b = 1$.
\end{proof}

\begin{remark}
    Theorem \ref{thm:carleman} states that the Carleman estimate \eqref{eq:carleman} holds for $\lambda > \lambda_0$ and $s > s_0$ large enough. 
    However, the only requirements on these parameters inside the proof above is the validity of \eqref{eq:sl_large}, namely that $s \lambda \psil(t) \ge 1$ for any $t \in [0,T]$. 
    Then, with both choices of exponential weights and polynomial weights, this simply holds true for $\lambda \ge 1$ and $s \ge 1$.
\end{remark}

\section{Stability results}
\label{sec:stability}

In this section, we apply the Carleman estimate proved in Theorem \ref{thm:carleman} to deduce some conditional stability results for the backward inverse problem associated to the system \eqref{eq:u}--\eqref{eq:bc}.
Our first result is a conditional H\"older estimate for any time $0 < t_0 < T$.

\begin{theorem}
	\label{thm:holder_stab}
	Let $(u,v)$ be a strong solution to \eqref{eq:u}--\eqref{eq:bc} satisfying \eqref{eq:hp_reg}, and assume that \eqref{eq:hp_FG} holds. 
	Furthermore, let $M > 0$ be a positive constant such that \begin{equation}
		\label{eq:hp_cond_holder}
		\norm{(u(\cdot, 0), v(\cdot, 0))}_{\Hx2 \times \Hx1} \le M.
	\end{equation}
	Then, for any $\lambda > 0$ large enough there exists a constant $C_1 > 0$, depending only on $\lambda$, $T$, and the parameters of the system, such that for any $t_0 \in (0,T)$ the following H\"older stability estimate holds:
	\begin{equation}
		\label{eq:holder_stab}
		\norm{(u(\cdot, t_0), v(\cdot, t_0))}_{\Lx2 \times \Lx2} \le C_1 M^{1 - \theta(t_0)} \norm{(u(\cdot, T), v(\cdot, T))}_{\Hx2 \times \Hx1}^{\theta(t_0)},
	\end{equation}
	if $\norm{(u(\cdot, T), v(\cdot, T))}_{\Hx2 \times \Hx1}$ is sufficiently small, where 
	\begin{equation}
		\label{eq:theta_tzero}
		\theta(t_0) = \frac{e^{\lambda t_0} - 1}{3e^{\lambda T} - 2 e^{\lambda t_0} - 1} \in (0,1).
	\end{equation}
\end{theorem}



\begin{proof}
	The proof is inspired by the one of \cite[Theorem 7.2.1]{yamamoto}. 
	However, here we propose an improved version which allows us to avoid an additional logarithmic factor in the stability estimate.
	
	\textsc{First Step.} We start by taking $\lambda > 0$ large enough so that the Carleman estimate \eqref{eq:carleman} holds with the exponential weights $\phil(t) = \psil(t) = e^{\lambda t}$, namely
	\begin{align*}
		& \int_{Q_T} \left\{ \frac{1}{s \psil(t)} \left( \abs{\partial_t u}^2 + \abs{\partial_t v}^2 \right) + \lambda \left( \abs{\Delta u}^2 + \abs{\nabla v}^2 \right) + s \lambda^2 \psil(t) \left( \abs{u}^2 + \abs{v}^2 \right) \right\} e^{2s\phil(t)} \, \de x \, \de t \\
		& \quad \le C \left( \int_{Q_T} \left( \abs{F(x,t,u,v)}^2 + \abs{G(x,t,u,v)}^2 \right) e^{2s\phil(t)} \, \de x \, \de t 
		+ \sum_{k=1}^7 \abs{B_k} \right).
	\end{align*}
	By hypothesis \eqref{eq:hp_FG}, we know that 
	\begin{align*}
        & \int_{Q_T} \left( \abs{F(x,t,u,v)}^2 + \abs{G(x,t,u,v)}^2 \right) e^{2s\phil(t)} \, \de x \, \de t \\ 
		& \quad \le \int_{Q_T} K(t) \left( \norm{u}^2_{\Hx2} + \norm{v}^2_{\Hx1} \right) e^{2s \phil(t)} \, \de x \, \de t, \\
		& \quad \le \int_{Q_T} \left\{ K_\infty \left( \abs{\Delta u}^2 + \abs{\nabla v}^2 \right) + K_\infty s \left( \abs{u}^2 + \abs{v}^2 \right) \right\} e^{2s \phil(t)} \, \de x \, \de t,
	\end{align*}
	for some constant $K_\infty := \norm{K}_{\Lt\infty} > 0$ independent of $\lambda$ and $s$.
	In the last line, we also used that $s > 1$.  
	Then, up to choosing $\lambda > 0$ even larger, we can absorb the source term on the left-hand side.  Indeed, we have that
	\begin{align*}
		& \int_{Q_T} \bigg\{ \frac{1}{s \psil(t)} \left( \abs{\partial_t u}^2 + \abs{\partial_t v}^2 \right) + (\lambda - K_\infty) \left( \abs{\Delta u}^2 + \abs{\nabla v}^2 \right) \\
		& \qquad \quad + s (\lambda^2 \psil(t) - K_\infty) \left( \abs{u}^2 + \abs{v}^2 \right) \bigg\} e^{2s\phil(t)} \, \de x \, \de t
		\le C \sum_{k=1}^7 \abs{B_k}.
	\end{align*}
	Now, recalling that $\psil(t) = e^{\lambda t} \ge 1$ is strictly increasing, we can choose $\lambda > 0$ such that $\lambda - K_\infty > 1$ and, then, factor out $1/\psil(T)$ to infer that 
	\begin{align*}
		& \frac{1}{\psil(T)} \int_{Q_T} \left\{ \frac{1}{s} \left( \abs{\partial_t u}^2 + \abs{\partial_t v}^2 \right) + \left( \abs{\Delta u}^2 + \abs{\nabla v}^2 \right) + s \left( \abs{u}^2 + \abs{v}^2 \right) \right\} e^{2s\phil(t)} \, \de x \, \de t \\
		& \quad \le \int_{Q_T} \bigg\{ \frac{1}{s \psil(t)} \left( \abs{\partial_t u}^2 + \abs{\partial_t v}^2 \right) + (\lambda - K_\infty) \left( \abs{\Delta u}^2 + \abs{\nabla v}^2 \right) \\
		& \qquad \qquad \quad + s (\lambda^2 \psil(t) - K_\infty) \left( \abs{u}^2 + \abs{v}^2 \right) \bigg\} e^{2s\phil(t)} \, \de x \, \de t
		\le C \sum_{k=1}^7 \abs{B_k}.
	\end{align*}
	Then, we conclude that 
	\begin{equation}
		\label{eq:holdstab:est1}
		\begin{split}
		& \int_{Q_T} \left\{ \frac{1}{s} \left( \abs{\partial_t u}^2 + \abs{\partial_t v}^2 \right) + \left( \abs{\Delta u}^2 + \abs{\nabla v}^2 \right) + s \left( \abs{u}^2 + \abs{v}^2 \right) \right\} e^{2s\phil(t)} \, \de x \, \de t \\
		& \quad \le C_\lambda \sum_{k=1}^7 \abs{B_k},
		\end{split}
	\end{equation}
	for a constant $C_{\lambda} > 0$ depending also on $\lambda$ and $T$.
	From this point onward, we consider $\lambda > 0$ fixed so that \eqref{eq:holdstab:est1} holds. 
	In this way, the only free parameter left to choose is $s > 1$. 
	Next, we observe that the boundary terms on the right-hand side can be bounded in the following way:
	\begin{align*}
		\sum_{k=1}^7 \abs{B_k} & = 
		s \lambda \int_\Omega \abs*{\left[ e^{\lambda t} \abs{u}^2 e^{2s\phil(t)} \right]^0_T } \, \de x
		+ \int_\Omega \abs*{ \left[ \abs{\Delta u}^2 e^{2s \phil(t)} \right]^T_0 } \, \de x \\
		& \quad + s \lambda \int_\Omega \abs*{ \left[ e^{\lambda t} \abs{v}^2 e^{2s\phil(t)} \right]^0_T } \, \de x
		+ \int_\Omega \abs*{ \left[ \abs{\nabla v}^2 e^{2s \phil(t)} \right]^T_0 } \, \de x \\
		& \quad + \abramo{2}\int_\Omega \abs*{ \left[ \abs{\nabla u \cdot \nabla v} e^{2s \phil(t)} \right]^T_0 } \, \de x 
		+ \frac{\lambda}{2} \int_\Omega \abs*{  \left[ \abs{u}^2 e^{2s\phil(t)} \right]^0_T } \, \de x \\
		& \quad + \frac{\lambda}{2} \int_\Omega \abs*{ \left[ \abs{v}^2 e^{2s\phil(t)} \right]^0_T } \, \de x \\
		& \le C_\lambda s \, \bigg\{ \left( \norm{u(\cdot, 0)}^2_{\Hx2} + \norm{v(\cdot,0)}^2_{\Hx1} \right) e^{2s \phil(0)} \\
		& \qquad \qquad + \left( \norm{u(\cdot,T)}^2_{\Hx2} + \norm{v(\cdot,T)}^2_{\Hx1} \right) e^{2s \phil(T)} \bigg\},
	\end{align*}
	where $C_\lambda > 0$ is again a constant depending on $\lambda$ and $T$.
	Hence, by neglecting the second term on the left-hand side of \eqref{eq:holdstab:est1}, we get that 
	\begin{equation}
		\label{eq:holdstab:est2}
		\begin{split}
			& \int_{Q_T} \left\{ \frac{1}{s} \left( \abs{\partial_t u}^2 + \abs{\partial_t v}^2 \right) + s \left( \abs{u}^2 + \abs{v}^2 \right) \right\} e^{2s\phil(t)} \, \de x \, \de t \\
			& \quad \le C_\lambda s \, \bigg\{ \left( \norm{u(\cdot, 0)}^2_{\Hx2} + \norm{v(\cdot, 0)}^2_{\Hx1} \right) e^{2s \phil(0)} \\
			& \qquad \qquad \quad + \left( \norm{u(\cdot, T)}^2_{\Hx2} + \norm{v(\cdot, T)}^2_{\Hx1} \right) e^{2s \phil(T)} \bigg\}.
		\end{split}
	\end{equation}

	\textsc{Second Step.} We now fix $t_0 \in (0,T)$. 
	Then, by the regularity $u, v \in \HT 1 {\Lx2}$, we can use the Fundamental Theorem of Calculus, together with the fact that $\phil(t) = e^{\lambda t}$ and Young's inequality, to infer that 
	\begin{align*}
		& \int_\Omega \left( \abs{u(x,t_0)}^2 + \abs{v(x,t_0)}^2 \right) e^{2s \phil(t_0)} \, \de x \\
		& \quad = \int_\Omega \left( \abs{u(x,T)}^2 + \abs{v(x,T)}^2 \right) e^{2s \phil(T)} \, \de x 
		- \int_{t_0}^T \ddt \left( \int_\Omega \left( \abs{u(x,t)}^2 + \abs{v(x,t)}^2 \right) e^{2s \phil(t)} \, \de x \right) \, \de t \\
		& \quad = \int_\Omega \left( \abs{u(x,T)}^2 + \abs{v(x,T)}^2 \right) e^{2s \phil(T)} \, \de x \\
		& \qquad - \int_{t_0}^T \int_\Omega \left( 2 u \,\partial_t u + 2s \phil'(t) \abs{u}^2 + 2 v \, \partial_t v + 2s \phil'(t) \abs{v}^2 \right) e^{2s \phil(t)} \, \de x \, \de t \\
		& \quad \le \int_\Omega \left( \abs{u(x,T)}^2 + \abs{v(x,T)}^2 \right) e^{2s \phil(T)} \, \de x \\
		& \qquad + 2 \lambda e^{\lambda T} \int_0^T \int_\Omega \left( \sqrt{s} \abs{u} \frac{1}{\sqrt{s}} \abs{\partial_t u} + s \abs{u}^2 + \sqrt{s} \abs{v} \frac{1}{\sqrt{s}} \abs{\partial_t v} + s \abs{v}^2 \right) e^{2s \phil(t)} \, \de x \, \de t \\
		& \quad \le \left( \norm{u(\cdot, T)}^2_{\Lx2} + \norm{v(\cdot, T)}^2_{\Lx2} \right) e^{2s \phil(T)} \\
		& \qquad + C_{\lambda, T} \int_{Q_T} \left\{ \frac{1}{s} \left( \abs{\partial_t u}^2 + \abs{\partial_t v}^2 \right) + s \left( \abs{u}^2 + \abs{v}^2 \right) \right\} e^{2s\phil(t)} \, \de x \, \de t.
	\end{align*}
	Therefore, by using \eqref{eq:holdstab:est2}, we deduce that 
	\begin{equation}
		\label{eq:holdstab:est3}
		\begin{split}
		& \left( \norm{u(\cdot, t_0)}^2_{\Lx2} + \norm{v(\cdot, t_0)}^2_{\Lx2} \right) e^{2s \phil(t_0)} \\
		& \quad \le C_\lambda s \, \bigg\{ \left( \norm{u(\cdot, 0)}^2_{\Hx2} + \norm{v(\cdot, 0)}^2_{\Hx1} \right) e^{2s \phil(0)} \\
		& \qquad \qquad \quad + \left( \norm{u(\cdot, T)}^2_{\Hx2} + \norm{v(\cdot, T)}^2_{\Hx1} \right) e^{2s \phil(T)} \bigg\}.
		\end{split}
	\end{equation}
    
	\textsc{Third Step.} 
	Starting from \eqref{eq:holdstab:est3}, we are now in position to deduce the H\"older stability estimate \eqref{eq:holder_stab}.
	In order to do this, we introduce the following compact notation:
	\[
		D := \norm{(u(\cdot,T), v(\cdot,T))}_{\Hx2 \times \Hx1}. 
	\]
	Thus, by recalling also the hypothesis \eqref{eq:hp_cond_holder}, we can rewrite \eqref{eq:holdstab:est3} as
	\begin{equation}
		\label{eq:holdstab:est4}
		\norm{(u(\cdot, t_0), v(\cdot, t_0))}^2_{\Lx2 \times \Lx2} \le C s \left( D^2 e^{2 s (\phil(T) - \phil(t_0))} + M^2 e^{-2s (\phil(t_0) - \phil(0))} \right),
	\end{equation}
	where $C > 0$ is a constant depending on $\lambda$, $T$ and the parameters of the system, but not on $s$.
	Our goal is now to choose $s \gg 1$ large enough to deduce \eqref{eq:holder_stab}, provided that $D$ is sufficiently small.
	To further simplify the notation, we also set
	\[
		p := \phil(t_0) - \phil(0) > 0, \quad q := \phil(T) - \phil(t_0) > 0, \quad p + q = \phil(T) - \phil(0) > 0,
	\]
	so that \eqref{eq:holdstab:est4} becomes
	\begin{equation}
		\label{eq:holdstab:est5}
		\norm{(u(\cdot, t_0), v(\cdot, t_0))}^2_{\Lx2 \times \Lx2} \le C s \left( D^2 e^{2 s q} + M^2 e^{-2s p} \right).
	\end{equation}
	Now, since $p, q > 0$, we can choose $s \gg 1$ large enough so that 
	\[
		s \le e^{s q} \quad \text{and} \quad s \le e^{s p}.
	\]
	In this way we can further estimate the right-hand side of \eqref{eq:holdstab:est5} as
	\begin{equation}
		\label{eq:holdstab:est6}
		\norm{(u(\cdot, t_0), v(\cdot, t_0))}^2_{\Lx2 \times \Lx2} \le C \left( D^2 e^{3 s q} + M^2 e^{- s p} \right).
	\end{equation}
	Next, we now choose $s_* \gg 1$ such that 
	\[
		D^2 e^{3 s q} = M^2 e^{- s p},
	\]
	namely 
	\begin{equation}
		\label{eq:holdstab:sstar}
		s_* = \frac{1}{3q + p} \log \frac{M^2}{D^2}.
	\end{equation}
	We observe that $s_* \to +\infty$ as $D \to 0$. 
	Therefore, it is indeed possible to choose it large enough so that all the previous constraints hold, provided that $D$ is small enough.
	With this choice, we have that 
	\[
		M^2 e^{- s_* p} = D^2 e^{3 s_* q} = D^2 \left( \frac{M^2}{D^2} \right)^{\frac{3q}{3q + p}} = M^{2 \frac{3q}{3q + p}} D^{2 \frac{p}{3q + p}}.
	\]
	Hence, by also taking the square root on both sides, \eqref{eq:holdstab:est6} finally becomes 
	\begin{equation}
		\label{eq:holdstab:est7}
		\norm{(u(\cdot, t_0), v(\cdot, t_0))}_{\Lx2 \times \Lx2} \le C M^{\frac{3q}{3q + p}} D^{\frac{p}{3q + p}},
	\end{equation}
	which is exactly \eqref{eq:holder_stab} with
	\[
		\theta(t_0) := \frac{p}{3q + p} = \frac{\phil(t_0) - \phil(0)}{3 \phil(T) - 2 \phil(t_0) - \phil(0)} 
		= \frac{e^{\lambda t_0} - 1}{3e^{\lambda T} - 2 e^{\lambda t_0} - 1} \in (0,1).
	\]
	This concludes the proof of Theorem \ref{thm:holder_stab}.
\end{proof}

We observe that the H\"older exponent $\theta(t_0)$ in Theorem \ref{thm:holder_stab} degenerates as $t_0 \to 0$.
Thus, the estimate above does not provide any further information on the initial data. 
Indeed, only a weaker stability estimate of logarithmic type is available for $t_0=0$, as shown below.

\begin{theorem}
	\label{thm:log_stab}
	Let $(u,v)$ be a strong solution to \eqref{eq:u}--\eqref{eq:bc} satisfying \eqref{eq:hp_reg}, and assume that \eqref{eq:hp_FG} holds.
	Let $M > 0$, $\lambda > 0$ and $C_1 > 0$ be as in Theorem \ref{thm:holder_stab}.
	Furthermore, let $M_1 > 0$ be a positive constant such that
	\begin{equation}
		\label{eq:hp_cond_log}
		\norm{(u,v)}_{\HT 1 {\Lx2} \times \HT 1 {\Lx2}} \le M_1.
	\end{equation}
	Then, there exists a constant $C_2 > 0$ such that
	\begin{equation}
		\label{eq:log_stab}
		\norm{(u(\cdot, 0), v(\cdot, 0))}_{\Lx2 \times \Lx2} \le C_2 \left( \log \frac{M}{\norm{(u(\cdot, T), v(\cdot, T))}_{\Hx2 \times \Hx1}} \right)^{-\frac14},
	\end{equation}
	if $\norm{(u(\cdot, T), v(\cdot, T))}_{\Hx2 \times \Hx1}$ is sufficiently small, where the constant $C_2$ can be quantified as 
	\[
		C_2 = \frac{2 M_1^{3/4} C_1^{1/4}}{\beta^{1/4}}, 
        \quad \text{with} \quad \beta = \frac{\lambda}{3e^{\lambda T} - 3}.
	\]
\end{theorem}

\begin{proof}
	The proof follows the ideas contained in the argument of \cite[Theorem 3.5]{BCFLR2024} and \cite[Theorem 7.2.2]{yamamoto}. 
	In particular, we stress that this approach provides a finer estimate than that of \cite[Theorem 7.2.2]{yamamoto}, while also requiring weaker hypotheses on the regularity of the solutions $u$ and $v$. 
    For the sake of completeness, we report the argument in full detail.

	Our aim is to employ the H\"older stability estimate \eqref{eq:holder_stab} for $t \in (0,T)$ to deduce a weaker stability estimate for $t=0$.
	Hence, to simplify the notations, we denote 
	\[
		\eps := \frac{\norm{(u(\cdot, T), v(\cdot, T))}_{\Hx2 \times \Hx1}}{M},
	\]
	so that the H\"older stability estimate proved in Theorem \ref{thm:holder_stab} can be rewritten as
	\begin{equation}
		\label{eq:logstab:est1}
		\norm{(u(\cdot, t), v(\cdot, t))}^2_{\Lx2 \times \Lx2} \le C_1 M \eps^{\theta(t)},
	\end{equation}
	for any $t \in (0,T)$, if $\eps$ is small enough.
	Then, we observe that under our regularity assumptions, the function $t \in [0,T] \to \norm{(u(\cdot,t), v(\cdot,t))}^2_{\Lx2 \times \Lx2}$ is absolutely continuous. Then, we can apply the Fundamental Theorem of Calculus to say that for any $t \in (0,T)$
	\begin{align*}
		& \norm{(u(\cdot, t), v(\cdot, t))}^2_{\Lx2 \times \Lx2} \\
		& \quad = \norm{(u(\cdot, 0), v(\cdot, 0))}^2_{\Lx2 \times \Lx2} + \int_0^t \frac{\de}{\de s} \left( \norm{u(\cdot, s)}^2_{\Lx2} + \norm{v(\cdot, s)}^2_{\Lx2} \right) \, \de s \\
		& \quad = \norm{(u(\cdot, 0), v(\cdot, 0))}^2_{\Lx2 \times \Lx2} + \int_0^t \left( 2 ( \partial_s u(\cdot, s), u(\cdot, s) )_{\Lx2} + 2 ( \partial_s v(\cdot, s), v(\cdot, s) )_{\Lx2} \right) \, \de s.
	\end{align*}
	Consequently, by using Cauchy-Schwarz's inequality, the stability inequality \eqref{eq:logstab:est1} and the bounds \eqref{eq:hp_cond_holder} and \eqref{eq:hp_cond_log}, we deduce that 
	\begin{align*}
		& \norm{(u(\cdot, 0), v(\cdot, 0))}^2_{\Lx2 \times \Lx2} \\
		& \quad \le \norm{(u(\cdot, t), v(\cdot, t))}^2_{\Lx2 \times \Lx2} \\
		& \qquad + 2 \left( \int_0^t \norm{(\partial_s u(\cdot, s), \partial_s v(\cdot, s))}^2_{\Lx2 \times \Lx2} \, \de s \right)^{1/2} \left( \int_0^t \norm{(u(\cdot, s), v(\cdot, s))}^2_{\Lx2 \times \Lx2}  \, \de s \right)^{1/2} \\
		& \quad \le C_1^2 M^2 \eps^{2\theta(t)} + 2 M_1 \left( \int_0^t M_1 \cdot C_1 M \eps^{\theta(s)} \, \de s  \right)^{1/2},
	\end{align*}
	where $\theta(t)$ is the H\"older exponent given in Theorem \ref{thm:holder_stab}, and we also used the fact that we can bound $\norm{(u,v)}_{\CT 0 {\Lx2} \times \CT 0 {\Lx2}} \le M_1$ on one factor in the second integral. 
	Now, by working on the explicit expressions of $\theta(t)$, one can easily see that, for any $t \in [0,T]$, it is bounded below by a linear function, namely 
	\[
		\theta(t) = \frac{e^{\lambda t} - 1}{3e^{\lambda T} - 2 e^{\lambda t} - 1} \ge \frac{\lambda}{3e^{\lambda T} - 3} t := \beta t,
	\]
	where $\beta > 0$ is the one defined in the statement. 
	Indeed, we just used the sharp inequalities $e^{\lambda t} - 1 \ge \lambda t$ on the numerator and $3e^{\lambda T} - 2 e^{\lambda t} - 1 \le 3e^{\lambda T} - 3$ on the denominator, the latter following from $e^{\lambda t} > 1$.
	Therefore, since we are assuming $0 < \eps < 1$, we have that 
	\[ 
		\eps^{\theta(t)} \le \eps^{\beta t} \quad \text{for any $t \in (0,T)$.} 
	\]
	Hence, we can further estimate
	\begin{align*}
		\norm{(u(\cdot, 0), v(\cdot, 0))}^2_{\Lx2 \times \Lx2} & \le C_1^2 M^2 \eps^{2\beta t} + 2 M_1^{3/2} C_1^{1/2} \left( \int_0^t \eps^{\beta s} \, \de s  \right)^{1/2} \\
		& = C_1^2 M^2 \eps^{2\beta t} + 2 M_1^{3/2} C_1^{1/2} \left( \frac{\eps^{\beta t}}{\beta \log \eps} - \frac{1}{\beta \log \eps} \right)^{1/2} \\
		& = C_1^2 M^2 \eps^{2\beta t} + \frac{2 M_1^{3/2} C_1^{1/2}}{\beta^{1/2} \sqrt{\abs{\log \eps}}} (1 - \eps^{\beta t})^{1/2} := g(t).
	\end{align*}
	Now, we observe that this estimate holds for any $t \in (0,T)$, so we can make it sharper and independent of $t$ by minimising the function $g(t)$. Then, we compute its derivative and, recalling that $\log \eps < 0$ since $0 < \eps < 1$, we obtain that 
	\[ 
		g'(t) = - C_1^2 M^2 2 \beta \abs{\log \eps} \cdot \eps^{2\beta t} + M_1^{3/2} C_1^{1/2} \beta^{1/2} \sqrt{\abs{\log \eps}} \frac{\eps^{\beta t}}{(1-\eps^{\beta t})^{1/2}}. 
	\]
	If we call $x = \eps^{\beta t}$, $0 \ls x \ls 1$, we can find the critical points by solving the equation
	\[ 
		- C_1^2 M^2 2 \beta \abs{\log \eps} \cdot x^2 + M_1^{3/2} C_1^{1/2} \beta^{1/2} \sqrt{\abs{\log \eps}} \frac{x}{(1-x)^{1/2}} = 0, 
	\]
	which, after some manipulations, is equivalent to
	\begin{equation}
		\label{eq:xsolve}
		x \sqrt{1-x} = \frac{M_1^{3/2}}{2 \beta^{1/2} C_1^{3/2} M^2 \sqrt{\abs{\log \eps}}}.
	\end{equation}
	By analysing the function $f(x)=x \sqrt{1-x}$, we see that its graph is like the one in Figure \ref{fig:root_function}, with maximum value at $\frac{2\sqrt{3}}{9}$. Then, equation \eqref{eq:xsolve} has at least a solution if 
	\[ 
		\frac{M_1^{3/2}}{2 \beta^{1/2} C_1^{3/2} M^2 \sqrt{\abs{\log \eps}}} \le  \frac{2\sqrt{3}}{9}, 
	\]
	which is verified if $\eps$ is small enough. 
	Moreover, one can easily see that the solution corresponding to the minimum value of $g(t)$ is the smallest one, let us call it $\bar{x}$. 
    Such $\bar{x}$ is attained at some value $\bar{t}$, which should belong to the interval $[0,T]$, at least when $\eps$ is sufficiently small. 
    Indeed, for $\eps << 1$, one can see that $x\sqrt{1-x} = \eps^{\beta t} \sqrt{1 - \eps^{\beta t}} \approx \eps^{\beta t}$, so that equation \eqref{eq:xsolve} approximately corresponds to 
    \[
        \eps^{\beta t} = \frac{M_1^{3/2}}{2 \beta^{1/2} C_1^{3/2} M^2 \sqrt{\abs{\log \eps}}}, 
    \]
    which is solved by 
    \[
        \bar{t} \approx \frac{ \log \left( \frac{M_1^{3/2}}{2 \beta^{1/2} C_1^{3/2} M^2 \sqrt{\abs{\log \eps}}} \right) }{ \beta \log \eps } \to 0^+ \quad \hbox{as $\eps \to 0^+$.}
    \]
    Then, for any fixed $T$, if $\eps$ is suffiently small, one can always find a suitable $\bar{t} \in [0,T]$.
    We now have to estimate the value of $\bar{x}$ to complete our stability estimate.
	Hence, we can see that the increasing branch of the function $f$ lies between the two lines $y=x$ and $y=\frac{\sqrt{3}}{3}x$, therefore we can estimate the value of $\bar{x}$ with the solutions corresponding to the two lines, namely
	\begin{equation}
		\label{eq:xestimate}
		\frac{M_1^{3/2}}{2 \beta^{1/2} C_1^{3/2} M^2 \sqrt{\abs{\log \eps}}} \le \bar{x} \le \frac{\sqrt{3}M_1^{3/2}}{2 \beta^{1/2} C_1^{3/2} M^2 \sqrt{\abs{\log \eps}}}.
	\end{equation}

	\begin{figure}[t]
    	\centering
    	\includegraphics[scale=0.8]{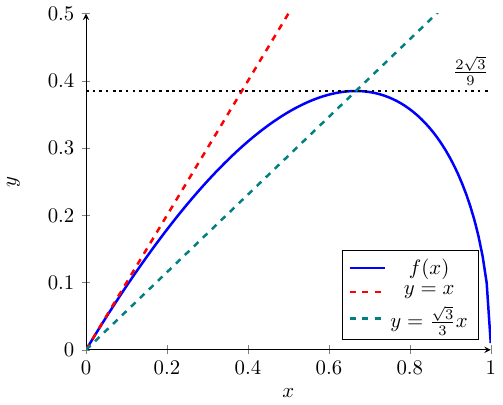}
    	\captionof{figure}{Graph of the function $f(x) = x \sqrt{1-x}$, \\ along with the two linear bounds.}
    	\label{fig:root_function}
	\end{figure}

	\noindent
	Finally, we can use this piece of information to close our estimate:
	\begin{align*}
		& \norm{(u(\cdot, 0), v(\cdot, 0))}^2_{\Lx2 \times \Lx2} \\
		& \quad \le \min_{t \in [0,T]} g(t) = C_1^2 M^2 \bar{x}^2 + \frac{2 M_1^{3/2} C_1^{1/2}}{\beta^{1/2} \sqrt{\abs{\log \eps}}} (1 - \bar{x})^{1/2} \\
		& \quad \le C_1^2 M^2 \frac{3 M_1^3}{4\beta C_1^3 M^4 \abs{\log \eps}} + \frac{2 M_1^{3/2} C_1^{1/2}}{\beta^{1/2} \sqrt{\abs{\log \eps}}} \underbrace{\left( 1 - \frac{M_1^{3/2}}{2 \beta^{1/2} C_1^{3/2} M^2 \sqrt{\abs{\log \eps}}} \right)^{1/2}}_{\le 1} \\
		& \quad \le \frac{2 M_1^{3/2} C_1^{1/2}}{\beta^{1/2}} \frac{1}{\sqrt{\abs{\log \eps}}} \left( 1 + \frac{3 M_1^{3/2}}{8 \beta^{1/2} C_1^{3/2} M^2 \sqrt{\abs{\log \eps}}} \right) \\
        & \quad \le \frac{4 M_1^{3/2} C_1^{1/2}}{\beta^{1/2}} \frac{1}{\sqrt{\abs{\log \eps}}},
	\end{align*}
	if we additionally assume that 
    \[
        \frac{3 M_1^{3/2}}{8 \beta^{1/2} C_1^{3/2} M^2 \sqrt{\abs{\log \eps}}} \le 1,
    \]
    which is also true if $\eps$ is small enough. 
	This concludes the proof of Theorem \ref{thm:log_stab}, by taking the square root on both sides and recalling that $\abs{\log \eps} = \log \frac{1}{\eps}$, if $0 < \eps < 1$.
\end{proof}

\section{Application to tumour growth models}
\label{sec:tumour}

    An important application of the results contained in Sections \ref{sec:carleman} and \ref{sec:stability} stems in the context of phase-field tumour growth models.
    Below, we show how we can apply our general results to a particular model, thus extending some previously obtained results in \cite{ABCFR2025}.
    However, we stress that the Carleman estimate and its consequences can be applied also to other phase-field tumour growth models, possibily with different source terms (e.g., \cite{GL2017, Fritz2023}), as long as they retain a Cahn--Hilliard-reaction-diffusion structure and admit sufficiently regular solutions (cf. Theorem \ref{thm:strongsols} below).
    
    We consider the following model, based on the one introduced in \cite{HZO2012}:
    \begin{alignat}{2}
    	& \partial_t \phi - \Delta \mu
    	= P(\phi) \left(\sigma + \chi (1-\phi) - \mu \right) - c(x,t) \hh (\phi)
    	\qquad && \hbox{in $Q_T$,} \label{eq:phi} \\
    	& \mu 
    	= - \Delta \phi + F'(\phi) - \chi \sigma 
    	\qquad && \hbox{in $Q_T$,} \label{eq:mu} \\
    	& \partial_t \sigma - \Delta \sigma + \chi \Delta \phi 
    	= - P(\phi) \left(\sigma + \chi (1-\phi) - \mu \right) + \kappa (1 - \sigma)
    	\qquad && \hbox{in $Q_T$,} \label{eq:sigma} \\
    	& \partial_{\n} \phi = \partial_{\n} \mu = \partial_{\n} \sigma = 0 
    	\qquad && \hbox{on $\Sigma_T$,} \label{eq:bct} \\
    	& \phi(0) = \phi_0, \quad \sigma(0) = \sigma_0 
    	\qquad && \hbox{in $\Omega$,} \label{eq:ict}
    \end{alignat}
    Here, $\phi$ is a phase-field variable representing the tumour volume fraction, where $\phi \equiv 0$ is the healthy phase and $\phi \equiv 1$ is the tumour one, and satisfies a Cahn--Hilliard-type equation, being $\mu$ the corresponding chemical potential. 
    Conversely, $\sigma$ represents the concentration of a key nutrient driving the tumour proliferation and satisfies a reaction-diffusion equation. 
    We also include chemotaxis effects through the parameter $\chi \ge 0$.
    The model describes tumour growth in the a-vascular phase and is derived within a diffuse-interface multiphase mixture approach. The dynamics of the tumour cells and the nutrient component are described by mass action laws, with associated chemical potentials containing cells-cells and cells-matrix adhesion contributions, encoded in the form of the nonlinearity $F(\phi)$, as well as interface, mass and chemotactic terms. See \cite{ABCFR2025,HPZO2013} for details. Also, the mass exchange terms between the phases are represented by chemical kinetic rates incorporating tumour cells proliferation, with phase-dependent proliferation rate encoded in the term $P(\phi)$, nutrient consumption, apoptosis and therapy-related death terms proportional to $\hh (\phi)$, with death rate $c(x,t)$, and nutrient release by the healthy tissues, with release rate $\kappa$.
    The inverse problem we want to address in this context is the identification of the initial data $(\phi_0, \sigma_0)$ starting from a single measurement of $(\phi(T), \sigma(T))$ at the final time.

    We start our analysis by specifying some hypotheses on the non-linearities and the parameters of the system, so that the direct problem is well-posed and admits a unique strong solution. 
    Indeed, we postulate the following assumptions (cf. \cite{ABCFR2025} and references therein).
    \begin{enumerate}[font = \bfseries, label = A\arabic*., ref = \bf{A\arabic*}]
	\item\label{ass:setting} $\Omega \subset \R^d$, $d=2,3$, is an open bounded domain with $C^4$ boundary, $T > 0$, and $\chi \ge 0$. 
	\item\label{ass:fbelow} $F \in C^4(\R)$ and there exist $c_1 \gs \chi^2 \ge 0$ and $c_2 \ge 0$ such that  
	\[ F(y) \ge c_1 y^2 - c_2 \quad \hbox{for any $y\in\R$.} \]
	\item\label{ass:fder} There exist $c_3 > 0$ and $c_4 \ge 0$ such that
	\[ \abs{F'(y)} \le c_3 F(y) + c_4 \quad \hbox{for any $y\in\R$.} \] 
	\item\label{ass:fconv} $F$ can be written as $F = F_0 + F_1$ for some $F_0, F_1 \in C^4(\R)$. Moreover, there exist $c_0, c_0' \gs 0$, $s \in [2,6)$ and $l \ge 0$ such that 
	\[ c_0' (1 + \abs{y}^{s-2}) \le F_0''(y) \le c_0 (1 + \abs{y}^{s-2}) \quad \hbox{for any $y\in\R$,} \]
	and 
	\[ \abs{F_1''(y)} \le l \quad \hbox{for any $y\in\R$.} \]
	\item\label{ass:p} $P \in C^2(\R)$ and there exist $c_5 >0$ and $q \in [1,2]$ such that 
	\[ 0 \le P(s) \le c_5 (1+\abs{y}^q) \quad \forall y \in \R. \]  
	\item\label{ass:hc} $\hh \in C^2(\R) \cap L^\infty(\R)$ and $c \in L^\infty(Q_T)$.
	\item\label{ass:initial} $\phi_0 \in \Hx2$ with $\partial_{\n} \phi_0 = 0$ on $\partial \Omega$, and $\sigma_0 \in \Hx1$.
    \end{enumerate}
    In the following, when saying that some constants will depend on the parameters of the model, we will mean that they will depend on some of the constants or some norms of the functions introduced above. 
    Moreover, for simplicity, we denote by 
    \[
        H^2_N(\Omega) := \{u \in \Hx2 \mid \partial_{\n} u = 0 \hbox{ on $\partial \Omega$}\}
    \]
    the Hilbert space in which we take the initial datum $\phi_0$ (cf. \ref{ass:initial}). 
    
    Under the above assumptions, in \cite{ABCFR2025} it was shown that the system \eqref{eq:phi}--\eqref{eq:ict} is well-posed in the sense of strong solutions (cf. also \cite{FGR2015_TumGrowth, GY2020}). 
    We report the main result below.

    \begin{theorem}
	\label{thm:strongsols}
	Under assumptions \ref{ass:setting}--\ref{ass:initial}, there exists a unique strong solution $(\phi, \mu, \sigma)$ to \eqref{eq:phi}--\eqref{eq:ict} enjoying the following regularities:
	\begin{align*}
		& \phi \in H^1(0,T; \Lx2) \cap C^0([0,T]; \Hx2) \cap L^2(0,T;\Hx 4), \\
		& \mu \in L^2(0,T; \Hx2), \\
		& \sigma \in H^1(0,T; \Lx2) \cap C^0([0,T]; \Hx1) \cap L^2(0,T; \Hx2), 
	\end{align*}
	In particular, there exists a constant $C>0$, depending only on the parameters of the model and on the data $\phi_0$ and $\sigma_0$, such that: 
	\begin{equation}
		\label{eq:strongnorms_est}
		\begin{split}
			& \norm{\phi}_{H^1(0,T;\Lx2) \cap L^\infty(0,T;\Hx2) \cap L^2(0,T;\Hx4)} 
			+ \norm{\mu}_{ L^2(0,T;\Hx2)} \\
			& \quad + \norm{\sigma}_{H^1(0,T;\Lx2) \cap L^\infty(0,T;\Hx1) \cap L^2(0,T;\Hx2)} \le C.
		\end{split}
	\end{equation}
    Moreover, a strong continuous dependence estimate holds.
    Indeed, let ${\phi_0}_1$, ${\sigma_0}_1$ and ${\phi_0}_2$, ${\sigma_0}_2$ be two sets of data satisfying \emph{\ref{ass:initial}} and let $(\phi_1, \mu_1, \sigma_1)$ and $(\phi_2, \mu_2, \sigma_2)$ two corresponding strong solutions as above. Then, there exists a constant $\bar{C}>0$, depending only on the data of the system and on the norms of $\{ ({\phi_0}_i, {\sigma_0}_i) \}_{i=1,2}$, but not on their difference, such that
	\begin{equation}
		\label{eq:contdep_estimate_strong}
		\begin{split}
			& \norm{\phi_1 - \phi_2}_{\HT 1 {\Lx2} \cap \LT \infty {\Hx2} \cap \LT 2 {\Hx4}} 
			+ \norm{\mu_1 - \mu_2}_{\LT 2 {\Hx2}} \\
			& \qquad + \norm{\sigma_1 - \sigma_2}_{\HT 1 {\Lx2} \cap \LT \infty {\Hx1} \cap \LT 2 {\Hx2}} \\
            & \quad \le \bar{C} \left( \norm{{\phi_0}_1 - {\phi_0}_2}_{\Hx2} + \norm{{\sigma_0}_1 - {\sigma_0}_2}_{\Hx1} \right).
		\end{split}
	\end{equation}
    \end{theorem}

    \begin{proof}
        See \cite[Theorem 3.2, Theorem 4.1 and Corollary 4.6] {ABCFR2025}
    \end{proof}

    \begin{remark}
    \label{rmk:fderivatives}
    By Sobolev embeddings, one can easily see that \eqref{eq:strongnorms_est} implies that 
    \[ \norm{\phi}_{\Cqt0} \le C, \]
    for some $C \gs 0$, depending only on the parameter of the system. 
    Consequently, given that $F \in C^4(\R)$ and $P, \hh \in C^2(\R)$, we also infer that
    \begin{gather*}
        \norm{F^{(i)}(\phi)}_{\Cqt0} \le C \quad \hbox{for any $i=1,\dots,4$,} \\
        \norm{P^{(i)}(\phi)}_{\Cqt0} \le C \quad \hbox{for any $i=0,\dots,2$,} \\
        \norm{\hh^{(i)}(\phi)}_{\Cqt0} \le C \quad \hbox{for any $i=0,\dots,2$.}
    \end{gather*}
    \end{remark}

    To better introduce our backward inverse problem, we define the following class of admissible initial data:
    \begin{equation}
        \label{eq:def:iad}
        \Iad = \{(\phi_0, \sigma_0) \in H^2_N(\Omega) \times \Hx1 \mid \norm{(\phi_0, \sigma_0)}_{\Hx2 \times \Hx1} \le M \},
    \end{equation}
    for some fixed constant $M > 0$.
    For later purposes, we also define the constant $M_1 > 0$ to be the sharpest one for which, given any $(\phi_0, \sigma_0) \in \Iad$, it holds 
    \begin{equation}
        \label{eq:def:m1}
        \norm{(\phi, \sigma)}_{\HT 1 {\Lx2} \times \HT 1 {\Lx2}} \le M_1,
    \end{equation}
    where $(\phi, \sigma)$ is the strong solution to \eqref{eq:phi}--\eqref{eq:ict}.
    Such a constant is related to the constant $C>0$ appearing on the right-hand side of \eqref{eq:strongnorms_est}.
    Both $M$ and $M_1$ will be the explicit constants appearing in the analogous versions of Theorems \ref{thm:holder_stab} and \ref{thm:log_stab} for our application to the tumour growth model \eqref{eq:phi}--\eqref{eq:ict}.

    We call $\Rcal: \Iad \to \Hx2 \times \Hx1$ the forward operator such that 
    \[
        \Rcal(\phi_0, \sigma_0) = (\phi(T), \sigma(T)),
    \]
    where $(\phi(T), \sigma(T))$ is the strong solution to \eqref{eq:phi}--\eqref{eq:ict}, evaluated at the final time.
    Then, by Theorem \ref{thm:strongsols}, we know that $\Rcal$ is well-defined and Lipschitz continuous. 
    As in \cite{ABCFR2025}, the inverse problem we are interested in concerns the determination of the earlier state $(\phi_0, \sigma_0)$, given a single measurement of the final condition $(\phi(T), \sigma(T))$. 
    Similar problems were considered also in \cite{BCFLR2024} for a specific prostate cancer model, and in \cite{FLS2021} for a non-local version of the above model \eqref{eq:phi}--\eqref{eq:ict}.
    In \cite[Theorem 4.3]{ABCFR2025}, it was shown that the solution to such an inverse problem for \eqref{eq:phi}--\eqref{eq:ict} is unique. 
    However, due to the approach based on logarithmic convexity for a fourth order system, no further stability results were obtained. 
    At the same time, an additional smallness condition on the chemotactic coefficient $\chi$ was required. 
    With the new approach based on Carleman estimates introduced in Sections \ref{sec:carleman} and \ref{sec:stability}, we plan to prove some stability results on the backward problem, by also removing the technical restriction on $\chi$ (cf. \cite[Remark 4.5]{ABCFR2025}).

    \subsection{Logarithmic stability}

    To apply the stability results above to this non-linear system, we consider two solutions $(\phi_1, \sigma_1)$ and $(\phi_2, \sigma_2)$ to \eqref{eq:phi}--\eqref{eq:ict} corresponding to two pairs of initial data $(\phi_0^i, \sigma_0^i)$ for $i=1,2$.
    Then, we let $\phitil = \phi_1 - \phi_2$, $\mutil = \mu_1 - \mu_2$, $\sigmatil = \sigma_1 - \sigma_2$, and, by substituting the corresponding equation for $\mutil$, we can write the system solved by the difference of two solutions in the following way:
    \begin{alignat}{2}
        & \partial_t \phitil + \Delta^2 \phitil + \chi \Delta \sigmatil = \tilde{f}_1 
        \quad && \hbox{in $Q_T$,} \label{eq:phi2} \\
        & \partial_t \sigmatil - \Delta \sigmatil + \chi \Delta \phitil = \tilde{f}_2 
        \quad && \hbox{in $Q_T$,} \label{eq:sigma2}\\
        & \partial_{\n} \phitil = \partial_{\n} \Delta \phitil = \partial_{\n} \sigmatil = 0 
        \quad && \hbox{on $\Sigma_T$,} \label{eq:bct2}\\
        & \phitil(0) = {\phi_0}_1 - {\phi_0}_2, \quad \sigmatil(0) = {\sigma_0}_1 - {\sigma_0}_2 
        \quad && \hbox{in $\Omega$,} \label{eq:ict2}
    \end{alignat}
    where the right-hand sides are
    \begin{align*}
        & \tilde{f}_1 = \Delta (F'(\phi_1)) - \Delta (F'(\phi_2)) + P(\phi_1)(\sigmatil - \chi \phitil + \Delta \phitil - (F'(\phi_1) - F'(\phi_2)) + \chi \sigmatil) \\
        & \qquad + (P(\phi_1) - P(\phi_2))(\sigma_2 + \chi (1 - \phi_2) + \Delta \phi_2 - F'(\phi_2) + \chi \sigma_2) - c (\hh(\phi_1) - \hh(\phi_2)), \\
        & \tilde{f}_2 = - P(\phi_1)(\sigmatil - \chi \phitil + \Delta \phitil - (F'(\phi_1) - F'(\phi_2)) + \chi \sigmatil) \\
        & \qquad - (P(\phi_1) - P(\phi_2))(\sigma_2 + \chi (1 - \phi_2) + \Delta \phi_2 - F'(\phi_2) + \chi \sigma_2) - \kappa \sigmatil.
    \end{align*}
    In particular, the boundary conditions \eqref{eq:bct} are equivalent to those in \eqref{eq:bct2}.
    Now, we observe that this system has the same structure of \eqref{eq:u}--\eqref{eq:bc}, with $\phitil = u$, $\sigmatil = v$, $a = b = \chi$, $F = \tilde{f}_1$ and $G = \tilde{f}_2$.
    Moreover, the following property holds.

    \begin{lemma}
        \label{lem:FG}
        Assume \ref{ass:setting}--\ref{ass:initial}, and let $(\phi_1, \sigma_1)$ and $(\phi_2, \sigma_2)$ be two strong solutions to \eqref{eq:phi}--\eqref{eq:ict} corresponding to two pairs of initial data $({\phi_0}_i, {\sigma_0}_i) \in \Iad$ for $i=1,2$.
        Then, the source terms $\tilde{f}_1$ and $\tilde{f}_2$ in \eqref{eq:phi2}--\eqref{eq:ict2} satisfy hypothesis \eqref{eq:hp_FG}, namely there exists a function $K \in \Lt \infty$ such that
        \begin{equation*}
            \norm{\tilde{f}_1}^2_{\Lx2} + \norm{\tilde{f}_2}^2_{\Lx2} \le K(t) \left( \norm{\phitil}^2_{\Hx2} + \norm{\sigmatil}^2_{\Hx1} \right),
        \end{equation*}
        for a.e. $t \in (0,T)$.
    \end{lemma}

    \begin{proof}
        By Young's inequality, we can start by estimating
        \begin{align*}
            & \norm{\tilde{f}_1}^2_{\Lx2} + \norm{\tilde{f}_2}^2_{\Lx2} \\
            & \quad \le C \Big( \norm{\Delta(F'(\phi_1)) - \Delta(F'(\phi_2))}^2_{\Lx2} \\
            & \qquad + 2 \norm{P(\phi_1)(\sigmatil - \chi \phitil + \Delta \phitil - (F'(\phi_1) - F'(\phi_2)) + \chi \sigmatil)}^2_{\Lx2} \\
            & \qquad + 2 \norm{(P(\phi_1) - P(\phi_2))(\sigma_2 + \chi (1 - \phi_2) + \Delta \phi_2 - F'(\phi_2) + \chi \sigma_2)}^2_{\Lx2} \\
            &\qquad + \norm{c (\hh(\phi_1) - \hh(\phi_2))}^2_{\Lx2} + \kappa^2 \norm{\sigmatil}^2_{\Lx2} \Big)\\
            & \quad = C \left(I_1 + I_2 + I_3 + I_4 + \norm{\sigmatil}^2_{\Lx2} \right).
        \end{align*}
        Now, by using the local Lipschitz continuity of $F''$ and $F'''$, guaranteed by $F \in C^4$, the fact that $\phi_1$ and $\phi_2$ are globally bounded by Remark \ref{rmk:fderivatives}, as well as Young and H\"older's inequalities, we see that
        \begin{align*}
            I_1 & \leq C \Big( \norm{F''(\phi_1) \Delta \phitil}^2_{\Lx2} + \norm{(F''(\phi_1) - F''(\phi_2)) \Delta \phi_2}^2_{\Lx2} \\
            & \quad + \norm{F'''(\phi_1) (\nabla \phi_1 + \nabla \phi_2) \cdot \nabla \phitil}^2_{\Lx2} +
            \norm{(F'''(\phi_1) - F'''(\phi_2)) \nabla \phi_2 \cdot \nabla \phi_2}^2_{\Lx2} \Big) \\
            & \le C \Big( \norm{F''(\phi_1)}^2_{\Lx\infty} \norm{\Delta \phitil}^2_{\Lx2} + \norm{\Delta \phi_2}^2_{\Lx2} \norm{\phitil}^2_{\Lx\infty} \\
            & \quad + \norm{F'''(\phi_1)}^2_{\Lx\infty} \norm{\nabla \phi_1 + \nabla \phi_2}^2_{\Lx4} \norm{\nabla \phitil}^2_{\Lx4} + \norm{\nabla \phi_2}^4_{\Lx4} \norm{\phitil}^2_{\Lx\infty} \Big) \\
            & \le C \Big( \norm{\Delta \phitil}^2_{\Lx2} + \norm{\Delta \phi_2}^2_{\Lx2} \norm{\phitil}^2_{\Hx2} \\ 
            & \quad + \norm{\phi_1 + \phi_2}^2_{\Hx2} \norm{\phitil}^2_{\Hx2} + \norm{\phi_2}^4_{\Hx2} \norm{\phitil}^2_{\Hx2} \Big) \\
            & \le C \left( 1 + \norm{\phi_2}^2_{\Hx2} + \norm{\phi_1}^2_{\Hx2} + \norm{\phi_2}^4_{\Hx2} \right) \norm{\phitil}^2_{\Hx2},
        \end{align*}
        where we extensively used the Sobolev embeddings $\Hx2 \hookrightarrow \Wx{1,4}$ and $\Hx2 \hookrightarrow \Lx\infty$. 
        Moreover, we also estimate similarly the remaining terms: 
        \begin{align*}
            I_2 & \le 2 \norm{P(\phi_1)}^2_{\Lx\infty} \norm{\sigmatil - \chi \phitil + \Delta \phitil - (F'(\phi_1) - F'(\phi_2)) + \chi \sigmatil}^2_{\Lx2} \\
            & \le C \left( \norm{\sigmatil}^2_{\Lx2} +  \norm{\phitil}^2_{\Lx2} + \norm{\Delta \phitil}^2_{\Lx2} \right) \\
            & \le C \left( \norm{\sigmatil}^2_{\Lx2} + \norm{\phitil}^2_{\Hx2} \right), \\
            I_3 & \le 2 \norm{(P(\phi_1) - P(\phi_2)}^2_{\Lx\infty} \norm{\sigma_2 + \chi (1 - \phi_2) + \Delta \phi_2 - F'(\phi_2) + \chi \sigma_2}^2_{\Lx2} \\
            & \le C \norm{\phitil}^2_{\Lx\infty} \left( 1 + \norm{\sigma_2}^2_{\Lx2} + \norm{\phi_2}^2_{\Lx2} + \norm{\phi_2}^2_{\Hx2} \right) \\
            & \le C \left( 1 + \norm{\sigma_2}^2_{\Lx2} + \norm{\phi_2}^2_{\Hx2} \right) \norm{\phitil}^2_{\Hx2}, \\
            I_4 & \le \norm{c}^2_{\Lqt\infty} \norm{\hh(\phi_1) - \hh(\phi_2)}^2_{\Lx2} \le C \norm{\phitil}^2_{\Lx2}.
        \end{align*}
        Putting it all together, we get that 
        \[ 
            \norm{\tilde{f}_1}^2_{\Lx2} + \norm{\tilde{f}_2}^2_{\Lx2} \le C \left( 1 + \norm{\sigma_2}^2_{\Lx2} + \norm{\phi_1}^2_{\Hx2} + \norm{\phi_2}^4_{\Hx2} \right) \left( \norm{\sigma}^2_{\Lx2} + \norm{\phi}^2_{\Hx2} \right),
        \]
        where 
        \[ 
            K(t) := C \left( 1 + \norm{\sigma_2}^2_{\Lx2} + \norm{\phi_1}^2_{\Hx2} + \norm{\phi_2}^4_{\Hx2} \right) \in \Lt\infty,
        \]
        since $\phi_i \in \LT \infty {\Hx2}$ and $\sigma_i \in \LT \infty {\Lx2}$, $i=1,2$, by Theorem \ref{thm:strongsols}. 
        Then, Lemma \ref{lem:FG} holds with $K$ as above.
    \end{proof}

    Then, by essentially exploiting the general Carleman estimate \eqref{eq:carleman}, which now holds also for the particular system \eqref{eq:phi2}--\eqref{eq:ict2}, we can prove the following results.

    \begin{theorem}
	\label{thm:holder_stab_tum}
	Assume \ref{ass:setting}--\ref{ass:initial}, and let $(\phi_1, \sigma_1)$ and $(\phi_2, \sigma_2)$ be two strong solutions to \eqref{eq:phi}--\eqref{eq:ict} corresponding to two pairs of initial data $({\phi_0}_i, {\sigma_0}_i) \in \Iad$ for $i=1,2$.
	Then, for any $\lambda > 0$ large enough there exists a constant $C_1 > 0$, depending only on $\lambda$, $T$, and the parameters of the system, such that for any $t_0 \in (0,T)$ the following H\"older stability estimate holds:
	\begin{equation}
		\label{eq:holder_stab_tum}
        \begin{split}
		& \norm{(\phi_1(\cdot, t_0), \sigma_1(\cdot, t_0)) - (\phi_2(\cdot, t_0), \sigma_2(\cdot, t_0))}_{\Lx2 \times \Lx2} \\
        & \quad \le C_1 M^{1 - \theta(t_0)} \norm{(\phi_1(\cdot, T), \sigma_1(\cdot, T)) - (\phi_2(\cdot, T), \sigma_2(\cdot, T))}_{\Hx2 \times \Hx1}^{\theta(t_0)},
        \end{split}
	\end{equation}
	if $\norm{(\phi_1(\cdot, T), \sigma_1(\cdot, T)) - (\phi_2(\cdot, T), \sigma_2(\cdot, T))}_{\Hx2 \times \Hx1}$ is sufficiently small, where 
	\begin{equation*}
		\theta(t_0) = \frac{e^{\lambda t_0} - 1}{3e^{\lambda T} - 2 e^{\lambda t_0} - 1} \in (0,1),
	\end{equation*}
    and $M > 0$ is the constant appearing in the definition \eqref{eq:def:iad} of $\Iad$.
    \end{theorem}

    \begin{proof}
        The result follows by application of Theorem \ref{thm:holder_stab}, together with Lemma \ref{lem:FG}.
    \end{proof}

    We stress that \eqref{eq:holder_stab_tum} provides a stability estimate for the backward problem for any positive time $t_0 \in (0,T)$, but does not immediately give any new information at time $t_0 = 0$, as the H\"older exponent degenerates. 
    However, we can still get a backward uniqueness result by continuity, as done in the following Corollary.

    \begin{corollary}
	\label{cor:backuniq}
    	Assume \ref{ass:setting}--\ref{ass:initial}, and let $(\phi_1, \sigma_1)$ and $(\phi_2, \sigma_2)$ be two strong solutions to \eqref{eq:phi}--\eqref{eq:ict} corresponding to two pairs of initial data $({\phi_0}_i, {\sigma_0}_i) \in \Iad$ for $i=1,2$.
    	
    	If $(\phi_1(\cdot,T), \sigma_1(\cdot,T)) = (\phi_2(\cdot,T), \sigma_2(\cdot,T))$, then $(\phi_1(\cdot,t), \sigma_1(\cdot,t)) = (\phi_2(\cdot,t), \sigma_2(\cdot,t))$ for any $t \in [0,T]$. In particular, $({\phi_0}_1, {\sigma_0}_1) = ({\phi_0}_2, {\sigma_0}_2)$ in $\Hx2 \times \Hx1$.
    \end{corollary}
    
    \begin{proof}
        Observe that, if $(\phi_1(\cdot,T), \sigma_1(\cdot,T)) = (\phi_2(\cdot,T), \sigma_2(\cdot,T))$, then estimate \eqref{eq:holder_stab_tum} gives that 
        \[ 
            \norm{(\phi_1(\cdot,t), \sigma_1(\cdot,t)) - (\phi_2(\cdot,t), \sigma_2(\cdot,t))}_{\Lx2 \times \Lx2} = 0
        \]
        for any $t \in (0,T]$.
        The result now easily follows also for $t=0$, by exploiting the continuity of the function $t \to \norm{(\phi_1(\cdot,t), \sigma_1(\cdot,t)) - (\phi_2(\cdot,t), \sigma_2(\cdot,t))}_{\Lx2 \times \Lx2}$, which is guaranteed by the $\CT 0 {\Lx2}$-regularity of the solutions. 
    \end{proof}

    \begin{remark}
        This backward uniqueness result improves the previous one obtained in \cite[Theorem 4.3]{ABCFR2025} by means of the logarithmic convexity method. 
        Indeed, here we no longer require an additional smallness constraint on the chemotaxis coefficient $\chi$, beyond the standard one needed for the existence of a solution (cf. \ref{ass:fbelow}).
        This is due to our new approach by Carleman estimates, which does not require any additional assumption on the cross-diffusion terms. 
        Moreover, this new approach also allows us to get   some explicit stability estimates for the backward problem, as we show below.
    \end{remark}

    \begin{theorem}
	\label{thm:log_stab_tum}
	Assume \ref{ass:setting}--\ref{ass:initial}, and let $(\phi_1, \sigma_1)$ and $(\phi_2, \sigma_2)$ be two strong solutions to \eqref{eq:phi}--\eqref{eq:ict} corresponding to two pairs of initial data $({\phi_0}_i, {\sigma_0}_i) \in \Iad$ for $i=1,2$.
	Let $M > 0$, $\lambda > 0$ and $C_1 > 0$ be as in Theorem \ref{thm:holder_stab_tum}.
	Furthermore, let $M_1 > 0$ be as in \eqref{eq:def:m1}.
	Then, there exists a constant $C_2 > 0$ such that
	\begin{equation}
		\label{eq:log_stab_tum}
        \begin{split}
		& \norm{({\phi_0}_1, {\sigma_0}_1) - ({\phi_0}_2, {\sigma_0}_2)}_{\Lx2 \times \Lx2} \\
        & \quad \le C_2 \left( \log \frac{M}{\norm{(\phi_1(\cdot, T), \sigma_1(\cdot, T)) - (\phi_2(\cdot, T), \sigma_2(\cdot, T))}_{\Hx2 \times \Hx1}} \right)^{-\frac14},
        \end{split}
	\end{equation}
	if $\norm{(\phi_1(\cdot, T), \sigma_1(\cdot, T)) - (\phi_2(\cdot, T), \sigma_2(\cdot, T))}_{\Hx2 \times \Hx1}$ is sufficiently small, where the constant $C_2$ can be quantified as 
	\[
		C_2 = \frac{2 M_1^{3/4} C_1^{1/4}}{\beta^{1/4}}, \quad \text{with} \quad \beta = \frac{\lambda}{3e^{\lambda T} - 3}.
	\]
    \end{theorem}

    \begin{proof}
        The result follows by application of Theorem \ref{thm:log_stab}, together with Lemma \ref{lem:FG}.
    \end{proof}

    \begin{remark}
    	We observe that this type of logarithmic estimate can be impractical in applications, as the error actually gets small only if the final data are very close to each other.  
        Moreover, it can be easily seen that the stability constant $C_2$ deteriorates exponentially with the final time $T$, making this stability estimate even more unreliable. 
        Nevertheless, this is expected since we are dealing with a backward problem for a parabolic system, which is known to be severely ill-posed as $T$ gets larger.
    \end{remark}

    \subsection{Further stability results}
    
    We now want to improve this stability estimate to one of Lipschitz type, under the additional a priori assumption that the initial data to be reconstructed actually lives in a finite-dimensional subspace of $\Hx2 \times \Hx1$, for instance, one of the discrete spaces used in numerical approximations. 
    To do this, we apply a general result, which was proved in \cite[Proposition 5]{BV2006}. 
    We note that the same result was successfully used to get Lipschitz stability estimates for other inverse ill-posed problems, like in \cite{AV2005, ABFV2022}, as well as for a similar backward problem on a specific prostate tumour growth model in \cite{BCFLR2024}.
    Proving this kind of refined Lipschitz stability estimate is crucial, as it paves the way to the implementation of robust numerical algorithms for the reconstruction (cf. \cite{BCFLR2025} and references therein).
    The main requirements to apply \cite[Proposition 5]{BV2006} are the $C^1$-regularity (in the sense of Fr\'echet) and the local injectivity of the forward map $\Rcal$.
    We devote the rest of this section to a brief exposition on how to show the validity of these properties, as their derivation is essentially standard once a strong well-posedness result (cf.~Theorem \ref{thm:strongsols}), the contributions contained in \cite[Section 5]{ABCFR2025} and the analysis carried out in Section \ref{sec:stability} are available.
    In particular, regarding the $C^1$-regularity of the forward map, we also refer the reader to \cite{CRW2021, CGRS2017, CSS2021_secondorder}, where the authors performed similar computations on a highly related system in the context of optimal control problems. 

    \subsubsection{Regularity of the forward map}

    By Theorem \ref{thm:strongsols}, we already know that $\Rcal$ is Lipschitz continuous. 
    Our aim is now to show that $\Rcal$ is continuously Fr\'echet differentiable. 
    First, as an \emph{ansatz} for the Fr\'echet derivative, we introduce the corresponding linearised system to \eqref{eq:phi}--\eqref{eq:ict}, which takes the form
    \begin{alignat}{2}
    	& \partial_t \xi - \Delta \eta 
    	= P'(\phib) (\sigmab + \chi (1 -\phib) - \mub) \xi + P(\phib) (\rho - \chi \xi - \eta) - \hh'(\phib) c \, \xi 
    	\qquad && \hbox{in  $Q_T$},  \label{eq:xi} \\
    	& \eta 
    	= - \Delta \xi + F''(\phib) \xi - \chi \rho 
    	\qquad && \hbox{in  $Q_T$},  \label{eq:eta} \\
    	& \partial_t \rho - \Delta \rho + \chi \Delta \xi
    	= - P'(\phib) (\sigmab + \chi (1 -\phib) - \mub) \xi - P(\phib) (\rho - \chi \xi - \eta) - \kappa \rho 
    	\qquad && \hbox{in  $Q_T$}, \label{eq:rho} \\
    	& \partial_{\n} \xi = \partial_{\n} \eta = \partial_{\n} \rho = 0 
    	\qquad && \hbox{on  $\Sigma_T$}, \label{eq:bcl} \\
    	& \xi(0) = h, \quad \rho(0) = k 
    	\qquad && \hbox{in  $\Omega$}, \label{eq:icl}
    \end{alignat}
    where $(h,k) \in \Hx2 \times \Hx1$ are increments and $(\phib, \mub, \sigmab)$ is the strong solution to \eqref{eq:phi}--\eqref{eq:ict}.
    Then, we state a strong well-posedness result for the linearised system \eqref{eq:xi}--\eqref{eq:icl}.
    
    \begin{proposition}
    	\label{prop:linearised}
    	Assume hypotheses \ref{ass:setting}--\ref{ass:initial}. Let $(\phib, \mub, \sigmab)$ be the strong solution to \eqref{eq:phi}--\eqref{eq:ict}, corresponding to some $(\phi_0, \sigma_0) \in \Iad$. Then, for any $(h,k) \in H^2_N(\Omega) \times \Hx1$, the system \eqref{eq:xi}--\eqref{eq:icl} admits a unique strong solution, which satisfies the system almost everywhere and is uniformly bounded in the following spaces
    	\begin{align*}
    		& \xi \in H^1(0,T;\Lx2) \cap C^0([0,T]; \Hx2) \cap L^2(0,T;\Hx4), \\
    		& \eta \in L^2(0,T;\Hx2), \\
    		& \rho \in H^1(0,T;\Lx2) \cap C^0([0,T]; \Hx1) \cap L^2(0,T;\Hx2).
    	\end{align*}
    \end{proposition}

    \begin{proof}
        The existence of a unique weak solution to \eqref{eq:xi}--\eqref{eq:icl} was shown in \cite[Theorem 5.3]{ABCFR2025}.
        Indeed, it was shown that, if $(\phib, \mub, \sigmab)$ is a weak solution to the forward system \eqref{eq:phi}--\eqref{eq:ict}, there exists a unique weak solution $(\xi, \eta, \rho)$, which satisfies \eqref{eq:xi}--\eqref{eq:icl} in a variational sense and is uniformly bounded in the following spaces
        \begin{align*}
    		& \xi \in H^1(0,T;(H^2_N(\Omega))^*) \cap C^0([0,T]; \Lx2) \cap L^2(0,T;H^2_N(\Omega)), \\
    		& \eta \in L^2(0,T;\Lx2), \\
    		& \rho \in H^1(0,T;(\Hx1)^*) \cap C^0([0,T]; \Lx2) \cap L^2(0,T;\Hx1),
    	\end{align*}
        namely there exists a constant $C > 0$, depending only on the parameters of the systems and on $(\phib, \mub, \sigmab)$, such that
        \begin{equation}
            \label{eq:lin:weakest}
            \begin{split}
                & \norm{\xi}_{H^1(0,T;(H^2_N(\Omega))^*) \cap L^\infty(0,T; \Lx2) \cap L^2(0,T;\Hx2)} 
                + \norm{\eta}_{L^2(0,T;\Lx2)} \\
                & \quad + \norm{\rho}_{H^1(0,T;(\Hx1)^*) \cap L^\infty(0,T; \Lx2) \cap L^2(0,T;\Hx1)} 
                \le C \left( \norm{h}_{\Lx2} + \norm{k}_{\Lx2} \right).
            \end{split}
        \end{equation}

        We now assume that $(\phib, \mub, \sigmab)$ is a strong solution to \eqref{eq:phi}--\eqref{eq:ict} corresponding to $(\phi_0, \sigma_0) \in \Iad$ and that $(h,k) \in H^2_N(\Omega) \times \Hx1$, then we show that $(\xi, \eta, \rho)$ enjoy stronger regularity properties. 
        For ease of presentation, we just argue by formal \emph{a priori} estimates.
        However, we stress that the estimates below can be made rigorous in a Galerkin discretisation scheme. 
        
    
        First, we test \eqref{eq:rho} by $\partial_t \rho - \Delta \rho$ to get that 
        \begin{align*}
            & \norm{\partial_t \rho}^2_{\Lx2} + \ddt \norm{\nabla \rho}^2_H + \norm{\Delta \rho}^2_{\Lx2} \\
            & \quad = - \chi (\Delta \xi, \partial_t \rho - \Delta \rho)_{\Lx2} - (P'(\phib) (\sigmab + \chi(1-\phib) - \mub) \xi, \partial_t \rho - \Delta \rho)_{\Lx2} \\
            & \qquad - (P(\phib)(\rho - \chi \xi - \eta), \partial_t \rho - \Delta \rho)_{\Lx2} - \kappa (\rho, \partial_t \rho - \Delta \rho)_{\Lx2}. 
        \end{align*}
        Then, by exploiting H\"older and Young's inequalities and Remark \ref{rmk:fderivatives}, we estimate the terms on the right-hand side and we infer that
        \begin{align*}
            & \norm{\partial_t \rho}^2_{\Lx2} + \ddt \norm{\nabla \rho}^2_{\Lx2} + \norm{\Delta \rho}^2_{\Lx2} \\
            & \quad \le \chi \norm{\Delta \xi}_{\Lx2} \norm{\partial_t \rho - \Delta \rho}_{\Lx2} \\
            & \qquad + \norm{P'(\phib)}_{\Lx\infty} \norm{\sigmab + \chi(1-\phib) - \mub}_{\Lx\infty} \norm{\xi}_{\Lx2} \norm{\partial_t \rho - \Delta \rho}_{\Lx2} \\
            & \qquad + \norm{P(\phib)}_{\Lx\infty} \norm{\rho - \chi \xi - \eta}_{\Lx2} \norm{\partial_t \rho - \Delta \rho}_{\Lx2} 
            + \kappa \norm{\rho}_{\Lx2} \norm{\partial_t \rho - \Delta \rho}_{\Lx2} \\
            & \quad \le \mezzo \norm{\partial_t \rho}^2_{\Lx2} + \mezzo \norm{\Delta \rho}^2_{\Lx2} + C \left( 1 + \norm{\sigmab + \chi(1-\phib) - \mub}^2_{\Lx\infty} \right) \norm{\xi}^2_{\Lx2} \\
            & \qquad + C \norm{\Delta \xi}^2_{\Lx2} + C \norm{\rho}^2_{\Lx2} + C \norm{\eta}^2_{\Lx2}, 
        \end{align*}
        where $\sigmab + \chi(1-\phib) - \mub$ is uniformly bounded in $\LT 2 {\Lx\infty}$, due to the strong regularity given by Theorem \ref{thm:strongsols} and the Sobolev embedding $\Hx2 \hookrightarrow \Lx\infty$. 
        Hence, by integrating on $(0,t)$, for any $t \in (0,T)$, and  applying \eqref{eq:lin:weakest} and elliptic regularity theory, we conclude that  
        \begin{equation}
            \label{eq:linear:rhostrong}
            \norm{\rho}^2_{\HT 1 {\Lx2} \cap \LT \infty {\Hx1} \cap \LT 2 {\Hx2}} \le C \left( \norm{h}^2_{\Lx2} + \norm{k}^2_{\Hx1} \right).
        \end{equation}
        At this point, by computing explicitly $\Delta \eta$ in strong formulation, we formally rewrite equations \eqref{eq:xi} and \eqref{eq:eta} as 
        \begin{equation}
            \label{eq:xi2}
            \begin{split}
                &\partial_t \xi + \Delta^2 \xi + \chi \Delta \rho 
                = F''(\phib) \Delta \xi + 2 F'''(\phib) \nabla \phib \cdot \nabla \xi + F^{(4)}(\phib) \nabla \phib \cdot \nabla \phib \, \xi \\
                & \quad + F'''(\phib)\xi \Delta \phib + P'(\phib) (\sigmab + \chi(1-\phib) - \mub) \xi + P(\phib)(\rho - \chi \xi - \eta) - \hh'(\phib) c \, \xi. 
            \end{split} 
        \end{equation}
        Then, we test \eqref{eq:xi2} by $\partial_t \xi + \Delta^2 \xi$ and we get
        \begin{align*}
            & \norm{\partial_t \xi}^2_{\Lx2} + \ddt \norm{\Delta \xi}^2_{\Lx2} + \norm{\Delta^2 \xi}^2_{\Lx2} = - \chi (\Delta \rho, \partial_t \xi + \Delta^2 \xi)_{\Lx2} \\
            & \quad + (F''(\phib) \Delta \xi + 2 F'''(\phib) \nabla \phib \cdot \nabla \xi + F^{(4)}(\phib) \nabla \phib \cdot \nabla \phib \, \xi + F'''(\phib)\xi \Delta \phib, \partial_t \xi + \Delta^2 \xi)_{\Lx2} \\ 
            & \quad + (P'(\phib) (\sigmab + \chi(1-\phib) - \mub) \xi + P(\phib)(\rho - \chi \xi - \eta) - \hh'(\phib) c \xi, \partial_t \xi + \Delta^2 \xi)_{\Lx2}
        \end{align*}
        We now proceed to estimate the terms on the right-hand side. 
        For the first one, we use Cauchy--Schwarz and Young's inequalities to infer that
        \[
            \chi (\Delta \rho, \partial_t \xi + \Delta^2 \xi)_{\Lx2} \le \frac{1}{4} \norm{\partial_t \xi}^2_{\Lx2} + \frac{1}{4} \norm{\Delta^2 \xi}^2_{\Lx2} + C \norm{\Delta \rho}^2_{\Lx2},
        \]
        where, by \eqref{eq:linear:rhostrong}, we know that $\rho$ is uniformly bounded in $\LT 2 {\Hx2}$. 
        Next, for the second term, we exploit H\"older and Young's inequalities and the Sobolev embeddings $\Hx2 \hookrightarrow \Wx{1,4}$ and $\Hx2 \hookrightarrow \Lx\infty$, together Remark \ref{rmk:fderivatives}, to deduce that 
        \begin{align*}
            & (F''(\phib) \Delta \xi + 2 F'''(\phib) \nabla \phib \cdot \nabla \xi + F^{(4)}(\phib) \nabla \phib \cdot \nabla \phib \, \xi + F'''(\phib)\xi \Delta \phib, \partial_t \xi + \Delta^2 \xi)_{\Lx2} \\
            & \quad \le \Big( \norm{F''(\phib)}_{\Lx\infty} \norm{\Delta \xi}_{\Lx2} 
            + 2 \norm{F'''(\phib)}_{\Lx\infty} \norm{\nabla \phib}_{\Lx4} \norm{\nabla \xi}_{\Lx4} \\
            & \qquad + \norm{F^{(4)}(\phib)}_{\Lx\infty} \norm{\nabla \phi}^2_{\Lx4} \norm{\xi}_{\Lx\infty} \\
            & \qquad + \norm{F'''(\phib)}_{\Lx\infty} \norm{\Delta \phib}_{\Lx2} \norm{\xi}_{\Lx\infty} \Big) 
            \left( \norm{\partial_t \xi}_{\Lx2} + \norm{\Delta^2 \xi}_{\Lx2} \right) \\
            & \quad \le C \left( \norm{\Delta \xi}_{\Lx2} + \norm{\phib}_{\Hx2} \norm{\xi}_{\Hx2} + \norm{\phib}^2_{\Hx2} \norm{\xi}_{\Hx2} \right) \left( \norm{\partial_t \xi}_{\Lx2} + \norm{\Delta^2 \xi}_{\Lx2} \right) \\
            & \quad \le \frac{1}{4} \norm{\partial_t \xi}^2_{\Lx2} + \frac{1}{4} \norm{\Delta^2 \xi}^2_{\Lx2} + C \left( 1 + \norm{\phib}^2_{\Hx2} + \norm{\phib}^4_{\Hx2} \right) \norm{\xi}^2_{\Hx2},
        \end{align*}
        where we recall that $\phib$ is uniformly bounded in $\LT \infty {\Hx2}$ by Theorem \ref{thm:strongsols}. 
        Finally, we estimate the last term similarly to what was done to deduce \eqref{eq:linear:rhostrong}, indeed we have that
        \begin{align*}
            & (P'(\phib) (\sigmab + \chi(1-\phib) - \mub) \xi + P(\phib)(\rho - \chi \xi - \eta) - \hh'(\phib) c \, \xi, \partial_t \xi + \Delta^2 \xi)_{\Lx2} \\
            & \quad \le \norm{P'(\phib)}_{\Lx\infty} \norm{\sigmab + \chi(1-\phib) - \mub}_{\Lx\infty} \norm{\xi}_{\Lx2} \norm{\partial_t \xi + \Delta^2 \xi}_{\Lx2} \\
            & \qquad + \norm{P(\phib)}_{\Lx\infty} \norm{\rho - \chi \xi - \eta}_{\Lx2} \norm{\partial_t \xi - \Delta \xi}_{\Lx2} \\
            & \qquad - c_\infty \norm{\hh'(\phib)}_{\Lx\infty} \norm{\xi}_{\Lx2} \norm{\partial_t \xi + \Delta^2 \xi}_{\Lx2} \\
            & \quad \le \frac{1}{4} \norm{\partial_t \xi}^2_{\Lx2} + \frac{1}{4} \norm{\Delta^2 \xi}^2_{\Lx2} + C \left( 1 + \norm{\sigmab + \chi(1-\phib) - \mub}^2_{\Lx\infty} \right) \norm{\xi}^2_{\Lx2} \\
            & \qquad + C \norm{\rho}^2_{\Lx2} + C \norm{\eta}^2_{\Lx2},
        \end{align*}
        where again $\norm{\sigmab + \chi(1-\phib) - \mub}^2_{\Lx\infty} \in \Lt1$ by Theorem \ref{thm:strongsols}. 
        Hence, by gathering all estimates and integrating on $(0,t)$, for any $t \in (0,T)$, we obtain the following:
        \begin{align*}
            & \int_0^t \norm{\partial_t \xi}^2_{\Lx2} \, \de s + \frac{1}{4} \norm{\Delta \xi (\cdot,t)}^2_{\Lx2} + \int_0^t \norm{\Delta^2 \xi}^2_{\Lx2} \, \de s \\
            & \quad \le \frac{1}{4} \norm{h}^2_{\Hx2} + C \int_0^T \left( 1 + \norm{\sigmab + \chi(1-\phib) - \mub}^2_{\Lx\infty} \right) \norm{\xi}^2_{\Lx2} \, \de s \\
            & \qquad + C \int_0^T \left( 1 + \norm{\phib}^2_{\Hx2} + \norm{\phib}^4_{\Hx2} \right) \norm{\xi}^2_{\Hx2} \, \de s + C \int_0^T \norm{\rho}^2_{\Hx2} \, \de s + C \int_0^T \norm{\eta}^2_{\Lx2} \, \de s
        \end{align*}
        Therefore, by applying Gronwall's inequality and elliptic regularity theory, together with the previous estimates \eqref{eq:lin:weakest} and \eqref{eq:linear:rhostrong}, we deduce that
        \begin{equation}
            \label{eq:linear:xistrong}
            \norm{\xi}^2_{\HT 1 {\Lx2} \cap \LT \infty {\Hx2} \cap \LT 2 {\Hx4}} \le C \left( \norm{h}^2_{\Hx2} + \norm{k}^2_{\Hx1} \right).
        \end{equation}
        Moreover, by comparison in \eqref{eq:eta} and ellptic regularity theory, the newly-found estimate \eqref{eq:linear:xistrong} easily implies also that
        \[
            \norm{\eta}^2_{\LT 2 {\Hx2}} \le C \left( \norm{h}^2_{\Hx2} + \norm{k}^2_{\Hx1} \right).
        \]
        This concludes the proof of Proposition \ref{prop:linearised}.
    \end{proof}

    Next, we show that the map $\Rcal$ is continuously Fr\'echet differentiable and characterise its derivative as the solution to the linearised system.

    \begin{proposition}
    	\label{prop:frechet}
    	Assume hypotheses \ref{ass:setting}--\ref{ass:initial}. Then, $\Rcal: H^2_N(\Omega) \times \Hx1 \to \Lx2 \times \Lx2$ is Fr\'echet differentiable, namely for any $(\phiob, \sigmaob) \in H^2_N(\Omega) \times \Hx1$ there exists a unique Fr\'echet derivative $\D\Rcal(\phiob, \sigmaob) \in \mathcal{L}(H^2_N(\Omega) \times \Hx1, \Lx2 \times \Lx2)$ such that, as $\norm{(h,k)}_{H^2_N(\Omega) \times \Hx1} \to 0$,
    	\begin{equation*}
    		\frac{ \norm{ \Rcal(\phiob + h, \sigmaob + k) - \mathcal{R}(\phiob, \sigmaob) - \D\mathcal{R}(\phiob, \sigmaob)[h,k]  }_{\Lx2 \times \Lx2}}{ \norm{(h,k)}_{H^2_N(\Omega) \times \Hx1} } \to 0.
    	\end{equation*}
    	Moreover, for any $(h,k) \in H^2_N(\Omega) \times \Hx1$, the Fr\'echet derivative at $(\phiob, \sigmaob)$ in $(h,k)$ is defined as 
    	\[ \D\mathcal{R}(\phiob, \sigmaob)[h,k] = (\xi(\cdot,T), \rho(\cdot,T))  \]
    	where $\xi(\cdot,T)$ and $\rho(\cdot,T)$ are the solutions to the linearised system \eqref{eq:xi}--\eqref{eq:icl} with initial data $(h,k)$, evaluated at the final time. 
    
        Furthermore, the Fr\'echet-derivative $\D\Rcal$ is Lipschitz-continuous as a function from $H^2_N(\Omega) \times \Hx1$ to the space $\mathcal{L}(H^2_N(\Omega) \times \Hx1, \Lx2 \times \Lx2)$. More precisely, the following estimate holds:
        \begin{equation}
            \label{eq:lipconst:frechet}
            \begin{split}
            & \norm{\D\mathcal{R}((\phiob_1, \sigmaob_1)) - \D\mathcal{R}((\phiob_2, \sigmaob_2))}^2_{\mathcal{L}(H^2_N(\Omega) \times \Hx1, \Lx2 \times \Lx2)} \\
            & \quad \le C_0 \norm{(\phiob_1, \sigmaob_1) - (\phiob_2, \sigmaob_2)}^2_{\Hnx2 \times \Hx1},
            \end{split}
        \end{equation}
        for some constant $C_0 > 0$ depending only on the parameters of the system. 
        This implies that the operator $\Rcal$ is of class $C^1$ between $H^2_N(\Omega) \times \Hx1$ and $\Lx2 \times \Lx2$.
    \end{proposition}

    \begin{proof}
        The Fr\'echet differentiability of $\Rcal$ from the weaker space $\Hx1 \times \Lx2$ to $\Lx2 \times \Lx2$ was shown in \cite[Theorem 5.4]{ABCFR2025}, by using only the weak regularity of the solution to the forward problem, upon assuming some more restrictive hypotheses on the non-linearities $F$, $P$ and $\hh$ (cf.~\cite[Remark 5.1]{ABCFR2025}).
        However, since we can use the full regularity of the strong solution, these restrictions can be lifted, and the same differentiability result holds in our initial set of hypotheses. 
        The key point is that the strong regularity guarantees that the phase variable is uniformly bounded (cf.~Remark \ref{rmk:fderivatives}).
        Moreover, by composing $\Rcal:\Hx1 \times \Lx2\to \Lx2 \times \Lx2$ with the continuous embedding $\Hnx2 \times \Hx1 \hookrightarrow \Hx1 \times \Lx2$, we directly obtain by the chain rule the Fr\'echet differentiability of the forward map restricted to the stronger space $\Rcal:\Hnx2 \times \Hx1\to \Lx2 \times \Lx2$.
        Then, the first claim of Proposition \ref{prop:frechet} is proved, and we just need to prove the Lipschitz continuity of the Fr\'echet derivative.
        This amounts to prove that, given $(\phiob_i, \sigmaob_i)$, $i=1,2$, with corresponding solutions $(\phib_i, \sigmab_i, \mub_i)$, $i=1,2$, it holds that, for any $(h,k) \in \Hnx2 \times \Hx1$ with $\norm{(h,k)}_{\Hnx2 \times \Hx1}=1$,
        \begin{equation}
            \label{eq:lipfrechet:aim}
            \begin{split}
                & \norm{\D\Rcal(\phiob_1, \sigmaob_1)[(h,k)] - \D\Rcal(\phiob_2, \sigmaob_2)[(h,k)]}^2_{\Lx2 \times \Lx2} \\
                & \quad = \norm{(\xi_1(\cdot,T), \rho_1(\cdot,T)) - (\xi_2(\cdot,T), \rho_2(\cdot,T)}^2_{\Lx2 \times \Lx2} \\
                & \quad \le C \norm{(\phiob_1, \sigmaob_1) - (\phiob_2, \sigmaob_2)}^2_{\Hnx2 \times \Hx1},
            \end{split}
        \end{equation}
        which, by taking the supremum over $(h,k) \in \Hnx2 \times \Hx1$ with unitary norm, yields \eqref{eq:lipconst:frechet}.
        Here, $(\xi_i, \eta_i, \rho_i)$, $i=1,2$, denote the correspondind strong solutions to the linearised system \eqref{eq:xi}--\eqref{eq:icl}, with $(\phib_i, \mub_i, \sigmab_i)$, $i=1,2$, respectively, and the same $(h,k)$.

        We, then, consider the system solved by the differences $\xi := \xi_1 - \xi_2$, $\eta := \eta_1 - \eta_2$ and $\rho := \rho_1 - \rho_2$, which upon some reformulations takes the following form
        \begin{alignat}{2}
        	& \partial_t \xi - \Delta \eta \notag \\ 
        	& \quad = P'(\phib_1) (\sigmab_1 + \chi (1 -\phib_1) - \mub_1) \xi \notag \\
            & \qquad + (P'(\phib_1) - P'(\phib_2))(\sigmab_2 + \chi (1 -\phib_2) - \mub_2) \xi_2 \notag \\
            & \qquad + P'(\phib_2)((\sigmab_1 - \sigmab_2) - \chi (\phib_1 - \phib_2) - (\mub_1 - \mub_2)) \xi_2 \notag \\
            & \qquad + P(\phib_1) (\rho - \chi \xi - \eta) 
            + (P(\phib_1) - P(\phib_2))(\rho_2 - \chi \xi_2 - \eta_2) \notag \\
            & \qquad - \hh'(\phib_1) c \, \xi 
            - (\hh'(\phib_1) - \hh'(\phib_2)) c \, \xi_2
        	\qquad && \hbox{in  $Q_T$},  \label{eq:xilip} \\
        	& \eta 
        	= - \Delta \xi + F''(\phib_1) \xi + (F''(\phib_1) - F''(\phib_2)) \xi_2 - \chi \rho 
        	\qquad && \hbox{in  $Q_T$},  \label{eq:etalip} \\
        	& \partial_t \rho - \Delta \rho + \chi \Delta \xi
        	\notag \\
            & \quad = - P'(\phib_1) (\sigmab_1 + \chi (1 -\phib_1) - \mub_1) \xi \notag \\
            & \qquad - (P'(\phib_1) - P'(\phib_2))(\sigmab_2 + \chi (1 -\phib_2) - \mub_2) \xi_2 \notag \\
            & \qquad - P'(\phib_2)((\sigmab_1 - \sigmab_2) - \chi (\phib_1 - \phib_2) - (\mub_1 - \mub_2)) \xi_2 \notag \\
            & \qquad - P(\phib_1) (\rho - \chi \xi - \eta) 
            + (P(\phib_1) - P(\phib_2))(\rho_2 - \chi \xi_2 - \eta_2)
            - \kappa \rho
        	\qquad && \hbox{in  $Q_T$}, \label{eq:rholip} \\
        	& \partial_{\n} \xi = \partial_{\n} \eta = \partial_{\n} \rho = 0 
        	\qquad && \hbox{on  $\Sigma_T$}, \label{eq:bcllip} \\
        	& \xi(0) = 0, \quad \rho(0) = 0 
        	\qquad && \hbox{in  $\Omega$}. \label{eq:icllip}
        \end{alignat}
        Clearly, the system \eqref{eq:xilip}--\eqref{eq:icllip} is satisfied in strong formulation by Proposition \ref{prop:linearised}, and $(\xi, \eta, \rho)$ satisfy the regularities stated therein. 
        Hence, to prove \eqref{eq:lipfrechet:aim}, we need to perform some global estimates. 
        Indeed, we test \eqref{eq:xilip} by $\xi$, \eqref{eq:etalip} by $\Delta \xi$, \eqref{eq:rholip} by $\rho$, and we sum them up to obtain:
        \begin{equation}
        \label{eq:lipfrechet:est}
            \begin{split}
                & \mezzo \ddt \norm{\xi}^2_{\Lx2} 
                + \mezzo \ddt \norm{\rho}^2_{\Lx2}
                + \norm{\Delta \xi}^2_{\Lx2}
                + \norm{\nabla \rho}^2_{\Lx2} \\
                & \quad \le
                (F''(\phib_1)\xi, \Delta \xi)_{\Lx2}
                + ((F''(\phib_1) - F''(\phib_2))\xi_2, \Delta \xi)_{\Lx2} 
                - 2 \chi (\rho, \Delta \xi)_{\Lx2} \\
                & \qquad +(P'(\phib_1) (\sigmab_1 + \chi (1 -\phib_1) - \mub_1) \xi, \xi - \rho)_{\Lx2} \\
                & \qquad + (P'(\phib_1) - P'(\phib_2))(\sigmab_2 + \chi (1 -\phib_2) - \mub_2) \xi_2, \xi - \rho)_{\Lx2} \\
                & \qquad + (P'(\phib_2)((\sigmab_1 - \sigmab_2) - \chi (\phib_1 - \phib_2) - (\mub_1 - \mub_2)) \xi_2, \xi - \rho)_{\Lx2} \\
                & \qquad + (P(\phib_1) (\rho - \chi \xi - \eta), \xi - \rho)_{\Lx2} \\
                & \qquad + ((P(\phib_1) - P(\phib_2))(\rho_2 - \chi \xi_2 - \eta_2), \xi - \rho)_{\Lx2} \\
                & \qquad - (\hh'(\phib_1) c \, \xi, \xi)_{\Lx2} 
                - ((\hh'(\phib_1) - \hh'(\phib_2)) c \, \xi_2, \xi)_{\Lx2} 
                - \kappa \norm{\rho}^2_{\Lx2}.
            \end{split}
        \end{equation}
        We now proceed to estimate all the terms on the right-hand side of \eqref{eq:lipfrechet:est}. 
        In the estimates below we will make repeated use of Cauchy--Schwartz, H\"older and Young's inequalities, as well as Sobolev embeddings.
        For the first three terms, we see that 
        \begin{align*}
            & (F''(\phib_1)\xi, \Delta \xi)_{\Lx2}
                + ((F''(\phib_1) - F''(\phib_2))\xi_2, \Delta \xi)_{\Lx2} 
                - 2 \chi (\rho, \Delta \xi)_{\Lx2} \\
            & \quad \le \frac14 \norm{\Delta \xi}^2_{\Lx2} 
            + C \norm{\rho}^2_{\Lx2} 
            + C \norm{\xi}^2_{\Lx2}
            + C \norm{\xi_2}^2_{\Lx\infty} \norm{\phib_1 - \phib_2}^2_{\Lx2},
        \end{align*}
        where we used Remark \ref{rmk:fderivatives} and the local Lipschitz continuity of $F''$.
        Next, for the fourth, fifth and sixth terms, we have that 
        \begin{align*}
            & (P'(\phib_1) (\sigmab_1 + \chi (1 -\phib_1) - \mub_1) \xi, \xi - \rho)_{\Lx2} \\
            & \quad \le C \left( 1 + \norm{\sigmab_1 + \chi (1 -\phib_1) - \mub_1}^2_{\Lx\infty} \right) \norm{\xi}^2_{\Lx2} + C \norm{\rho}^2_{\Lx2} \\
            & (P'(\phib_1) - P'(\phib_2))(\sigmab_2 + \chi (1 -\phib_2) - \mub_2) \xi_2, \xi - \rho)_{\Lx2} \\
            & \quad \le C \norm{\sigmab_2 + \chi (1 -\phib_2) - \mub_2}^2_{\Lx\infty} \left( \norm{\xi}^2_{\Lx2} + \norm{\rho}^2_{\Lx2} \right) + C \norm{\xi_2}^2_{\Lx\infty} \norm{\phib_1 - \phib_2}^2_{\Lx2} \\
            & (P'(\phib_2)((\sigmab_1 - \sigmab_2) - \chi (\phib_1 - \phib_2) - (\mub_1 - \mub_2)) \xi_2, \xi - \rho)_{\Lx2} \\
            & \quad \le C \norm{\xi}^2_{\Lx2} + C \norm{\rho}^2_{\Lx2} \\
            & \qquad + C \norm{\xi_2}^2_{\Lx\infty} \left( \norm{\phib_1 - \phib_2}^2_{\Lx2} + \norm{\mub_1 - \mub_2}^2_{\Lx2} + \norm{\sigmab_1 - \sigmab_2}^2_{\Lx2} \right),
        \end{align*}
        where we used again Remark \ref{rmk:fderivatives} and the local Lipschitz continuity of $P'$.
        Next, to estimate the seventh term, we first observe that, by \eqref{eq:etalip}, it holds that 
        \begin{align*}
            \norm{\eta}_{\Lx2} \le
            C \norm{\Delta \xi}_{\Lx2} 
            + C \norm{\xi}_{\Lx2} 
            + \norm{\xi_2}_{\Lx\infty} \norm{\phib_1 - \phib_2}_{\Lx2} 
            + C \norm{\rho}_{\Lx2}.
        \end{align*}
        Then, we deduce that 
        \begin{align*}
            & (P(\phib_1) (\rho - \chi \xi - \eta), \xi - \rho)_{\Lx2} \\
            & \quad \le \frac14 \norm{\Delta \xi}^2_{\Lx2} + C \norm{\xi}^2_{\Lx2} + C \norm{\rho}^2_{\Lx2} + C \norm{\xi_2}^2_{\Lx\infty} \norm{\phib_1 - \phib_2}^2_{\Lx2}.
        \end{align*}
        Going on, for the eighth term, we similarly have that
        \begin{align*}
            & ((P(\phib_1) - P(\phib_2))(\rho_2 - \chi \xi_2 - \eta_2), \xi - \rho)_{\Lx2} \\
            & \quad C \norm{\rho_2 - \chi \xi_2 - \eta_2}^2_{\Lx\infty} \left( \norm{\xi}^2_{\Lx2} + \norm{\rho}^2_{\Lx2} \right) + C \norm{\phib_1 - \phib_2}^2_{\Lx2}.
        \end{align*}
        Finally, the ninth and tenth terms can be treated as
        \begin{align*}
            & (\hh'(\phib_1) c \, \xi, \xi)_{\Lx2} 
            + ((\hh'(\phib_1) - \hh'(\phib_2)) c \, \xi_2, \xi)_{\Lx2} \\
            & \quad \le C \norm{\xi}^2_{\Lx2} + C \norm{\xi_2}^2_{\Lx\infty} \norm{\phib_1 - \phib_2}^2_{\Lx2},
        \end{align*}
        by Remark \ref{rmk:fderivatives} and the local Lipschitz continuity of $\hh'$.
        By putting all together and integrating over $(0,t)$, for any $t \in (0,T)$, we conclude that
        \begin{align*}
            & \mezzo \norm{\xi(\cdot,t)}^2_{\Lx2} + \mezzo \norm{\rho(\cdot,t)}^2_{\Lx2} + \mezzo \int_0^t \norm{\Delta \xi}^2_{\Lx2} \, \de s + \int_0^t \norm{\nabla \rho}^2_{\Lx2} \, \de s \\
            & \quad \le C \int_0^T \left( 1 + \norm{\sigmab_1 + \chi (1 -\phib_1) - \mub_1}^2_{\Hx2} + \norm{\rho_2 - \chi \xi_2 - \eta_2}^2_{\Hx2} \right) \left( \norm{\xi}^2_{\Lx2} + \norm{\rho}^2_{\Lx2} \right) \, \de s \\
            & \qquad + C \norm{\xi_2}^2_{\LT \infty {\Hx2}} \int_0^T \left( \norm{\phib_1 - \phib_2}^2_{\Lx2} + \norm{\mub_1 - \mub_2}^2_{\Lx2} + \norm{\sigmab_1 - \sigmab_2}^2_{\Lx2} \right) \, \de s, 
        \end{align*}
        where $\norm{\sigmab_1 + \chi (1 -\phib_1) - \mub_1}^2_{\Hx2} \in \Lt1$, $\norm{\rho_2 - \chi \xi_2 - \eta_2}^2_{\Hx2} \in \Lt1$ and $\xi_2 \in \LT \infty {\Hx2}$ are uniformly bounded by Theorem \ref{thm:strongsols} and Proposition \ref{prop:linearised}. 
        Hence, by applying Gronwall's inequality, together with elliptic regularity theory and the continuous dependence estimate \eqref{eq:contdep_estimate_strong} on the second term on the right-hand side, we infer that 
        \begin{align*}
            & \norm{\xi}^2_{\LT \infty {\Lx2} \cap \LT 2 {\Hx2}} + \norm{\rho}^2_{\LT \infty {\Lx2} \cap \LT 2 {\Hx1}} \\
            & \quad \le C \left( \norm{\phiob_1 - \phiob_2}^2_{\Hnx2} + \norm{\sigmaob_1 - \sigmaob_2}^2_{\Hx1} \right).
        \end{align*}
        Moreover, by comparison in \eqref{eq:xilip} and \eqref{eq:rholip}, we also easily get that 
        \[
            \norm{\xi}^2_{\HT 1 {(\Hnx2)^*}} + \norm{\rho}^2_{\HT 1 {(\Hx1)^*}} \le C \left( \norm{\phiob_1 - \phiob_2}^2_{\Hnx2} + \norm{\sigmaob_1 - \sigmaob_2}^2_{\Hx1} \right).
        \]
        Then, by the standard embeddings $\HT 1 {(\Hnx2)^*} \cap \LT 2 {\Hx2} \hookrightarrow \CT 0 {\Lx2}$ and $\HT 1 {(\Hx1)^*} \cap \LT 2 {\Hx1} \hookrightarrow \CT 0 {\Lx2}$, we finally deduce that 
        \[
            \norm{\xi(\cdot,T)}^2_{\Lx2} + \norm{\rho(\cdot,T)}^2_{\Lx2} \le C \left( \norm{\phiob_1 - \phiob_2}^2_{\Hnx2} + \norm{\sigmaob_1 - \sigmaob_2}^2_{\Hx1} \right),
        \]
        which is \eqref{eq:lipfrechet:aim}.
        This concludes the proof of Proposition \ref{prop:frechet}.
        
    \end{proof}

    \subsubsection{Lipschitz stability}

    The last ingredient we need to prove a Lipschitz stability estimate for the inverse problem is the injectivity of the Fr\'echet derivative $\D \Rcal(\phiob, \sigmaob)$ for any fixed point $(\phiob, \sigmaob) \in \Hx2 \times \Hx1$. 
    By Proposition \ref{prop:frechet}, we know that $\D\mathcal{R}(\phiob, \sigmaob)[h,k] = (\xi(\cdot,T), \rho(\cdot,T))$ for any $(h,k) \in H^2(\Omega) \times \Hx1$, where $\xi(\cdot,T)$ and $\rho(\cdot,T)$ are the solutions to \eqref{eq:xi}--\eqref{eq:icl}. 
    Therefore, by linearity, showing that $\D \Rcal(\phiob, \sigmaob)$ is injective means proving that if $(\xi(\cdot,T), \rho(\cdot,T)) = (0,0)$ in $\Lx2 \times \Lx2$, then also $(h,k)=(0,0)$ in $H^2_N(\Omega) \times \Hx1$.
    Equivalently, this means showing a backward uniqueness result in the spirit of Theorem \ref{cor:backuniq} for the linearised system \eqref{eq:xi}--\eqref{eq:icl}. 
    However, the new stability result shown in Theorem \ref{thm:log_stab} for general Cahn--Hilliard-reaction-diffusion systems through Carleman estimates allows us to show a quantitative injectivity result, by providing a stability estimate for the backward problem on the linearised system. 
    Then, this allows us to use the general result in \cite[Proposition 5]{BV2006} to its full extent, by also providing an explicit estimate of the Lipschitz stability constant.
    Indeed, we have the following.

    \begin{proposition}
    	\label{prop:inizstab_lin}
    	Assume hypotheses \ref{ass:setting}--\ref{ass:initial}. Let $(\phib, \mub, \sigmab)$ be a strong solution to \eqref{eq:phi}--\eqref{eq:icl} corresponding to some initial data $(\phi_0, \sigma_0) \in \Iad$. Moreover, let $(\xi, \eta)$ be a strong solution to \eqref{eq:xi}--\eqref{eq:icl} corresponding to $(h,k) \in H^2_N(\Omega) \times \Hx1$. 
    	Assume also that 
        \[ \norm{(h,k)}_{H^2_N(\Omega) \times \Hx1} \le L, \]
        for some $L > 0$ and call $L_1 > 0$ the minimal constant such that, by Proposition \emph{\ref{prop:linearised}},
    	\begin{equation}
        \label{eq:def:LL1} 
            \norm{(\xi, \rho)}_{\HT 1 {\Lx2 \times \Lx2}} \le L_1. 
        \end{equation}
    	Then, there exists a constant $Q_1 \gs 0$, depending only on the parameters of the system, such that we have the following conditional stability estimate for any $t_0 \in [0,T]$
    	\[ 
            \norm{(\xi(\cdot,t_0), \rho(\cdot,t_0))}_{\Lx2 \times \Lx2} \le Q_1 L^{1-\theta(t_0)} \norm{(\xi(\cdot,T), \rho(\cdot,T))}_{\Hx2 \times \Hx1}^{\theta(t_0)},  
        \]
    	if $\norm{(\xi(\cdot,T), \rho(\cdot,T))}_{\Hx2 \times \Hx1}$ is sufficiently small, where $\theta(t_0)$ is the same exponent appearing in Theorem \ref{thm:holder_stab}.
    	
    	Moreover, there exists a constant $Q_2 > 0$ such that 
    	\begin{equation}
    		\label{eq:inizstab_lin}
    		\norm{ (h, k) }_{\Lx2 \times \Lx2} \le Q_2 \left( \log \frac{L}{\norm{(\xi(\cdot,T), \rho(\cdot,T))}_{\Hx2 \times \Hx1}} \right)^{-\frac14}, 
    	\end{equation}
    	if $\norm{(\xi(\cdot,T), \rho(\cdot,T))}_{\Hx2 \times \Hx1}$ is sufficiently small, where we can quantify the constant as 
    	\begin{equation}
            \label{eq:def:Q2} 
            Q_2 = \frac{2 L_1^{3/4} C_1^{1/4}}{\beta^{1/4}}, \quad \text{with} \quad \beta = \frac{\lambda}{3e^{\lambda T} - 3}.
        \end{equation}
    \end{proposition}

    \begin{proof}
        By rewriting the linearised system \eqref{eq:xi}--\eqref{eq:icl} only in terms of the independent variables $\xi$ and $\rho$, through the substitution of \eqref{eq:eta} into \eqref{eq:xi} and \eqref{eq:rho}, we get
        \begin{alignat*}{2}
        	& \partial_t \xi + \Delta^2 \xi - \chi \Delta \rho
        	= f_1 
        	\qquad && \hbox{in  $Q_T$}, \\
        	& \partial_t \rho - \Delta \rho + \chi \Delta \xi
        	= f_2 
        	\qquad && \hbox{in  $Q_T$}, \\
        	& \partial_{\n} \xi = \partial_{\n} \Delta \eta = \partial_{\n} \rho = 0 
        	\qquad && \hbox{on  $\Sigma_T$}, \\
        	& \xi(0) = h, \quad \rho(0) = k 
        	\qquad && \hbox{in  $\Omega$},
        \end{alignat*}
        where, upon some computations,  
        \begin{align*}
            f_1 & := F''(\phib) \Delta \xi + 2 F'''(\phib) \nabla \phib \cdot \nabla \xi + F'''(\phib) \Delta \phib \, \xi + F^{(4)} \nabla \phib \cdot \nabla \phib \, \xi \\
            & \quad + P'(\phib) (\sigmab + \chi (1 -\phib) - \mub) \xi + P(\phib) (\rho - \chi \xi + \Delta \xi - F''(\phib) \xi + \chi \rho) - \hh'(\phib) c \, \xi, \\
            f_2 & := - P'(\phib) (\sigmab + \chi (1 -\phib) - \mub) \xi - P(\phib) (\rho - \chi \xi + \Delta \xi - F''(\phib) \xi + \chi \rho) - \kappa \rho. 
        \end{align*}
        Thus, we observe that the above system has the same structure as the general system \eqref{eq:u}--\eqref{eq:bc}, with $\xi = u$, $\rho = v$, $a = b = \chi$, $F = f_1$ and $G = f_2$. 
        Moreover, one can easily verify that the source terms $f_1$ and $f_2$ satisfy hypothesis \eqref{eq:hp_FG}, by essentially repeating similar estimates to those performed in Lemma \ref{lem:FG}.
        Hence, we can apply Theorems \ref{thm:carleman}, \ref{thm:holder_stab} and \ref{thm:log_stab} to deduce the sought results on the linearised system.
    \end{proof}

    We are now in the position to apply \cite[Proposition 5]{BV2006} and establish a Lipschitz stability estimate, upon the additional restriction to a finite-dimensional subspace.
    For the reader's convenience, we report below the statement of the main tool we intend to use. 

    \begin{lemma}[\!\!\cite{BV2006}, Proposition 5]
    	\label{lem:vessella}
    	Let $R$ and $r_0$ be positive numbers. Let $\Lambda$ be an open subset of $\R^n$, for some $n \in \N$, and let $K$ be a subset of $\Lambda$. Assume that 
    	\begin{equation}
    		\label{eq:ass_finitedim}
    		K \subseteq B_R(0) \quad \text{and} \quad B_{r_0}(p) \subseteq \Lambda, \text{ for any } p\in K.
    	\end{equation}
    	Let $\mathcal{B}$ be a Banach space and $G: \Lambda \to \mathcal{B}$ an operator. Assume that $G$ satisfies the following conditions:
    	\begin{itemize}
    		\item[$(i)$] $G_{\mid K}$ is injective,
    		\item[$(ii)$] $(G_{\mid K})^{-1}: G(K) \to K$ is uniformly continuous with modulus of continuity $\omega_0(\cdot)$,
    		\item[$(iii)$] $G$ is Fr\'echet-differentiable with derivative $\D G$,
    		\item[$(iv)$] $\D G: \Lambda \to \mathcal{L}(\Lambda, \mathcal{B})$ is uniformly continuous with modulus of continuity $\omega_1(\cdot)$.
    		\item[$(v)$] there exists $m_0 \gs 0$ such that 
    			\[ \inf_{x \in K, \abs{y}=1} \norm{\D G(x)[y]}_{\mathcal{B}} \ge m_0. \]
    	\end{itemize}
    	Then, for any $x_1, x_2 \in K$
    	\begin{equation}
    		\abs{x_1 - x_2} \le C \norm{G(x_1) - G(x_2)}_{\mathcal{B}}, 
    	\end{equation}
    	where $C = \max\left\{ \frac{2R}{(\omega_0^{-1})_*(\delta_1)}, \frac{2}{m_0} \right\}$, $\delta_1 = \mezzo \min \{ \delta_0, r_0 \}$ and $\delta_0 = (\omega_1^{-1})_* \left( \frac{m_0}{2} \right)$. Here, for a modulus of continuity $\omega$, we denote $(\omega^{-1})_*(y) := \inf \,\{ x  \in [0, +\infty) \mid \omega(x) \ge y \}$.
    \end{lemma}

    \noindent
    Then, by applying Lemma \ref{lem:vessella}, we can prove the following quantitative Lipschitz stability result on the backward inverse problem.

    \begin{theorem}
    	\label{thm:lipstab}
    	Assume hypotheses \ref{ass:setting}--\ref{ass:initial}. Let $\Lambda$ be a finite-dimensional subspace of $H^2_N(\Omega) \times \Hx1$,
        and assume further that 
    	\[ (\phi_0, \sigma_0, p_0) \in \Iad \cap \Lambda. \]
    	Then, there exists a constant $C_s \gs 0$ such that for any choice of $({\phi_0}_1, {\sigma_0}_1)$ and $({\phi_0}_2, {\sigma_0}_2)$ in $\Iad \cap \Lambda$ the following stability estimate holds
    	\begin{equation}
    		\label{eq:lipstab}
            \begin{split}
                &\norm{ ({\phi_0}_1, {\sigma_0}_1) - ({\phi_0}_2, {\sigma_0}_2) }_{\Lx2 \times \Lx2} \\
                & \quad \le C_s \norm{ (\phi_1(\cdot,T), \sigma_1( \cdot,T)) - (\phi_2(\cdot,T), \sigma_2(\cdot,T)) }_{\Hx2 \times \Hx1},
            \end{split}
    	\end{equation}
    	where the constant $C_s$ can be quantified as 
    	\begin{gather*}
    	    C_s = \max \left\{ 2 e^{ \frac{16 C_0^2 C_2}{m_0}}, \frac{2}{m_0} \right\}, \\
            \text{with} \quad  
            m_0 = \frac{L}{C_{\Lambda}} e^{-Q^4_2} 
            \quad \text{and} \quad 
            C_{\Lambda} = \sup_{(h,k) \in \Lambda \setminus \{0\}} \frac{\norm{(h,k)}_{\Hx2 \times \Hx1}}{\norm{(h,k)}_{\Lx2 \times \Lx2}}, 
    	\end{gather*} 
    	where $C_0$ is the Lipschitz constant of the Fr\'echet derivative, given by \eqref{eq:lipconst:frechet}, and $M$, $L$, and $Q_2$ are defined in \eqref{eq:def:iad}, \eqref{eq:def:LL1}, \eqref{eq:def:Q2}.
    \end{theorem}
    
    \begin{proof}
        The proof follows the same strategy used in \cite[Theorem 3.10]{BCFLR2024} for a similar backward problem.
    	Indeed, we want to apply Lemma \ref{lem:vessella}, so we have to verify its hypotheses. 
        Hence, we set $X = H^2_N(\Omega) \times \Hx1$ and $\Lambda$ as above. Being a finite-dimensional subspace of $H^2_N(\Omega) \times \Hx1$ means that $\Lambda \simeq \R^n$ in the notation of Lemma \ref{lem:vessella}, therefore we can immediately see that $K := \Iad \cap \Lambda \subset \Lambda$ is a compact subset of $\Lambda$, so it satisfies hypothesis \eqref{eq:ass_finitedim}.
    	We consider the operator $\Rcal: \Lambda \subseteq H^2_N(\Omega) \times \Hx1 \to \Lx2 \times \Lx2$ defined as $\Rcal((\phi_0, \sigma_0)) = (\phi(\cdot,T), \sigma(\cdot,T))$. 
    	Hypothesis $(i)$ on the injectivity of $\Rcal_{\mid K}$ follows directly from Corollary \ref{cor:backuniq}. 
    	Also hypothesis $(ii)$ follows immediately from Theorem \ref{thm:log_stab_tum}, with modulus of continuity $\omega_0(r) = C_2 \left( \log (1/r) \right)^{-1/4}$. 
    	Then, hypotheses $(iii)$ and $(iv)$ on Fr\'echet differentiability follow from Proposition \ref{prop:frechet}, with modulus of continuity $\omega_1(r) = C_0 r$. Finally, we have to check condition $(v)$. 
    	To do this, notice that for any $(\phi_0, \sigma_0) \in K$ and any $(h,k) \in \Lambda$, we have that $\D\Rcal (\phi_0, \sigma_0) [h,k] = (\xi(\cdot,T), \rho(\cdot,T))$, where $(\xi, \rho)$ is the solution to \eqref{eq:xi}--\eqref{eq:icl}. 
    	Now, by Proposition \ref{prop:inizstab_lin}, we can use \eqref{eq:inizstab_lin} to say that, if $\norm{(h,k)}_{\Hx2 \times \Hx1} \le L$, it holds that
    	\[ 
            \norm{(h,k)}_{\Lx2 \times \Lx2} \le Q_2 \left( \log \frac{L}{\norm{(\xi(\cdot,T), \rho(\cdot,T))}_{\Hx2 \times \Hx1}} \right)^{-\frac14}, 
        \]
    	which implies that 
    	\[ 
            \norm{(\xi(\cdot,T), \rho(\cdot,T))}_{\Hx2 \times \Hx1} \ge L e^{- \frac{Q_2^4}{\norm{(h,k)}^4_{\Lx2 \times \Lx2}}}. 
        \]
    	Hence, we can verify condition $(v)$ as
    	\begin{align*}
    		& \inf_{\substack{(\phi_0, \sigma_0) \in K \\ (h,k) \in \Lambda \setminus \{0\}}} \frac{\norm{(\xi(\cdot,T), \rho(\cdot,T))}_{\Hx2 \times \Hx1}}{\norm{(h,k)}_{\Hx2 \times \Hx1}} 
    		\ge \inf_{\substack{(\phi_0, \sigma_0) \in K \\ (h,k) \in \Lambda \setminus \{0\}}} \frac{\norm{(\xi(\cdot,T), \rho(\cdot,T))}_{\Hx2 \times \Hx1}}{C_{\Lambda} \norm{(h,k)}_{\Lx2 \times \Lx2}} \\
    		& \quad = \inf_{\substack{(\phi_0, \sigma_0) \in K \\ \norm{(h,k)}_{\Lx2 \times \Lx2}=1 }} \frac{\norm{(\xi(\cdot,T), \rho(\cdot,T))}_{\Hx2 \times \Hx1}}{C_{\Lambda} \norm{(h,k)}_{\Lx2 \times \Lx2}}
    		\ge \frac{L}{C_{\Lambda}} e^{-Q_2^4}, 
    	\end{align*}
    	where we used the linearity of $\D\Rcal$ and we estimated
        \[
            \norm{(h,k)}_{\Hx2 \times \Hx1} \le C_{\Lambda} \norm{(h,k)}_{\Lx2 \times \Lx2},
        \]
        for any $(h,k) \in \Lambda$, since $\Lambda$ is finite-dimensional and thus all norms on it are equivalent.
    	
    	Then, we can apply Lemma \ref{lem:vessella} and conclude the proof of Theorem \ref{thm:lipstab}. 
    \end{proof}
    
    \begin{remark}
        We point out that the constant $C_s$ blows up exponentially as the dimension of $\Lambda$ goes to infinity. 
        Indeed, as one can see from its expression, $C_s$ has a direct exponential proportionality to the constant $C_\Lambda$, which is finite only if $\Lambda$ has a finite dimension and blows up as it becomes larger. 
        Additionally, one can also notice that the dependence of $C_s$ on the final time is even worse, namely $\log(C_s)$ depends exponentially on $T$. 
        This means that, even if we have some kind of Lipschitz stability, one has to be very careful when designing numerical algorithms to approximate the solution.
    \end{remark}

    \section{Conclusions}

    In the present work, we addressed the backward inverse problem for a coupled system of Cahn--Hilliard-reaction-diffusion equations with cross-diffusion (chemotaxis) and source terms, establishing a   novel Carleman estimate which provides the starting point to derive rigorous stability estimates for the inverse problem. 
    A key technical advance of our approach is that it removes the smallness restriction on the chemotaxis coefficients, which had been a limitation in earlier studies relying on logarithmic convexity arguments. This makes the framework more robust for biological applications, where such a restriction may be physically unrealistic.
    We derived H\"{o}lder-type stability estimates for the reconstruction of past states at any intermediate positive time $t_0 \in (0, T)$ and a logarithmic stability estimate for the initial datum at $t=0$. These results ensure that, given a single final-time observation at time $T$, earlier configurations of the system can be recovered with quantitative reliability under natural a priori bounds.
    
    We then apply these theoretical results to a phase-field tumour growth model widely studied in literature, which describes the tumour growth in the a-vascular phase and couples the tumour growth with a nutrient dynamics through chemotaxis and source and consumption terms. Thanks to the application of the Carleman estimate, we proved backward uniqueness and quantitative stability for the recovery of early tumour states given a single final time measurement.
    Furthermore, we established a Lipschitz stability result on finite-dimensional subspaces of admissible sets for the initial data, providing a solid mathematical foundation for the development of convergent iterative discretisation algorithms.

    The application of these theoretical results to phase-field tumour growth models offers significant clinical implications. Indeed, identifying the tumour's state at initiation or earlier stages of development can provide important diagnostic and prognostic informations to clinicians which may improve diagnostic accuracy and the design of personalised radiotherapy or surgical interventions. 
    For these reasons, future research will focus on the development of convergent and robust iterative discretised solvers for the backward inverse problem associated to phase-field tumour growth models, leveraging on efficient finite element approximations or neural network algorithms to solve the severely ill-posed inverse reconstruction problem under different clinically relevant situations.
    Further developments will concern the extension of the present theoretical framework to more complex models, in which the derivation of Carleman estimates and regularity results is more challenging, including models incorporating anisotropic effects or multi-species interactions.
    
    \bigskip
	
    \noindent\textbf{Acknowledgements.} This research is part of the activities of ``Dipartimento di Eccellenza 2023-2027'' of Universit\`a degli Studi di Milano (C.~Cavaterra and M. Fornoni). E.~Beretta has been partially supported by NYUAD Science Program Project Fund AD364.
    A. Agosti, C. Cavaterra and M. Fornoni are members of GNAMPA (Gruppo Nazionale per l’Analisi Matematica, la Probabilità e le loro Applicazioni) of INdAM (Istituto Nazionale di Alta Matematica). 

    \newpage 

    \footnotesize

\end{document}